\documentclass[microtype]{gtpart}
\usepackage[utf8]{inputenc}
\usepackage{amsmath,amsfonts}
\usepackage{amsthm}
\usepackage{amssymb}
\usepackage[parfill]{parskip}
\usepackage[english]{babel}
\usepackage{amsfonts}
\usepackage[T1]{fontenc}
\usepackage[utf8]{inputenc}
\usepackage{enumerate}
\usepackage{verbatim}
\usepackage{mathtools}
\usepackage{graphicx}
\usepackage{verbatim}
\usepackage[all]{xy}
\usepackage{tikz}
\usepackage{relsize}
\usepackage{amsmath}
\usepackage{amscd}
\usepackage{caption}
\usepackage{xcolor}
\usepackage{verbatim}
\usepackage[all]{xy}
\usepackage{tikz}
\usepackage{relsize}
\usepackage{amsmath}
\usepackage{amscd}
\usepackage{hyperref}
\usepackage[hyperpageref]{backref} 
\usepackage{caption}
\usepackage{xcolor}
\usepackage{textcomp}

\usetikzlibrary{knots}
\makeatletter
\def\thm@space@setup{%
  \thm@preskip=\parskip \thm@postskip=0pt
}
\makeatother
\newtheorem{teo}{Theorem}[section]
\newtheorem{lem}[teo]{Lemma}
\newtheorem{prop}[teo]{Proposition}
\newtheorem{cor}[teo]{Corollary}
\newtheorem*{teo*}{Theorem}

\newtheorem{cnj}{Conjecture}
\newtheorem*{teo:maxvol}{Theorem \ref{maxvol}}
\newtheorem*{teo:cong}{Theorem \ref{teo:cong}}
\newtheorem*{teo:maxvolconj}{Theorem \ref{teovol}}
\newtheorem*{lem:stima3}{Lemma \ref{stima3}}
\newtheorem*{lem:symm}{Lemma \ref{symm}}
\newtheorem*{prop:1}{Proposition \ref{prop:degen}}
\newtheorem*{prop:2}{Proposition \ref{prop:gener}}
\newtheorem*{prop:3}{Proposition \ref{prop:impropideal}}
\theoremstyle{definition}
\newtheorem{dfn}[teo]{Definition}
\newtheorem*{dfn*}{Definition}

\newtheorem{oss}[teo]{Remark}
\newtheorem*{dom*}{Question}
\newtheorem*{cng:maxvol}{The Maximum Volume Conjecture}

\numberwithin{equation}{section}

\newcommand{\ra}{\rightarrow}

\newcommand{\vol}{\mathrm{Vol}}

\DeclarePairedDelimiter\floor{\lfloor}{\rfloor}
\newcommand{\h}{\mathbb{H}^3}

\newcommand\abs[1]{\left|#1\right|}
\title{An upper bound conjecture for the Yokota invariant}
\author{Giulio Belletti}
\givenname{Giulio}
\surname{Belletti}
\address{B\^{a}timent 307 Rue Michel Magat, 91400 Orsay, France}
\email{gbelletti451@gmail.com}
\urladdr{https://sites.google.com/view/giulio-bellettis-homepage/}
\date{}
\begin{document}

\begin{abstract}
 We conjecture an upper bound on the growth of the Yokota invariant of polyhedral graphs, extending a previous result on the growth of the $6j$-symbol. Using Barrett's Fourier transform we are able to prove this conjecture in a large family of examples. As a consequence of this result, we prove the Turaev-Viro Volume Conjecture for a new infinite family of hyperbolic manifolds.
\end{abstract}

\maketitle

\tableofcontents
\section{Introduction}

In \cite{cyvolconj} Chen and Yang proposed and provided extensive computations for the following conjecture, relating the hyperbolic volume of a manifold to its Turaev-Viro invariants $TV_r$ (see \cite[Page 869]{TV} for the original definition):

\begin{cnj}[The Turaev-Viro Volume Conjecture]\label{volconj}
 Let $M$ be a hyperbolic $3$-manifold, either closed, with cusps, or compact with geodesic boundary. Then as $r$ varies along the odd natural numbers,
 \begin{equation}
 \lim_{r\ra\infty} \frac{2\pi}{r}\log\left(TV_r\left(M,e^{\frac{2\pi i}{r}}\right)\right)=\mathrm{Vol}(M)
 \end{equation}
\end{cnj}

This conjecture has been verified for the complements of the Borromean rings \cite{DKY}, of the figure eight knot \cite{DKY}, all the hyperbolic Dehn surgeries on the figure eight knot (for integral surgeries in \cite{ohtdf} and later for rational surgeries in \cite{wongyangvolume}), and all complements of fundamental shadow links \cite{bound6j}.

A useful tool introduced in \cite{bound6j} to study the asymptotic behavior of quantum invariants such as $TV_r$ is a sharp upper bound on the growth of the $6j$-symbol, which is the basic building block in their definition. Such an upper bound can be used to prove very quickly the Volume Conjecture for complements of fundamental shadow links.

The upper bound just mentioned can be interpreted as an upper bound for the Yokota invariant $Y_r$, which is an invariant of embedded graphs (see Definition \ref{dfn:Yok}). Indeed, the square of the $6j$-symbol is also the Yokota invariant of the tetrahedral graph; thus it is natural to ask if an upper bound analogous to the one of \cite{bound6j} holds for any polyhedral graph (which is to say, any graph which is the $1$-skeleton of a hyperbolic polyhedron). We propose the following:

\begin{cnj}[The Upper Bound Conjecture]\label{cng:upperbound}
  Let $r>2$ be odd. If $\Gamma$ is a polyhedral graph and $col$ is any $r$-admissible coloring of its edges (see Definition \ref{dfn:adm}), then
  \begin{displaymath}
   \frac{\pi}{r}\log\left\lvert Y_r(\Gamma,col)\right\rvert\leq \sup_{P}\mathrm{Vol}(P)+O_{r\ra\infty}\left(\frac{\log(r)}{r}\right)
  \end{displaymath}
  where $P$ varies among all proper generalized hyperbolic polyhedra with $1$-skeleton $\Gamma$ (see Definition \ref{dfn:poly}; these are hyperbolic polyhedra with possibly hyperideal vertices) and $\mathrm{Vol}(P)$ is the hyperbolic volume of $P$.
  
 Moreover, the inequality is sharp, with equality attained by the sequence of colorings
giving the color $\frac{r-2\pm1}{2}$ to each edge (the sign is chosen so that the colors are even).
\end{cnj}

We are able to prove the Upper Bound Conjecture for a large family of examples:

\begin{teo:maxvolconj}
 The Upper Bound Conjecture is verified for any planar graph obtained from the tetrahedron by applying any sequence of the following two moves:
 \begin{itemize}
  \item blowing up a trivalent vertex (see Figure \ref{fig:trunc}) or
  \item triangulating a triangular face (see Figure \ref{fig:triang}).
 \end{itemize} 
\end{teo:maxvolconj}

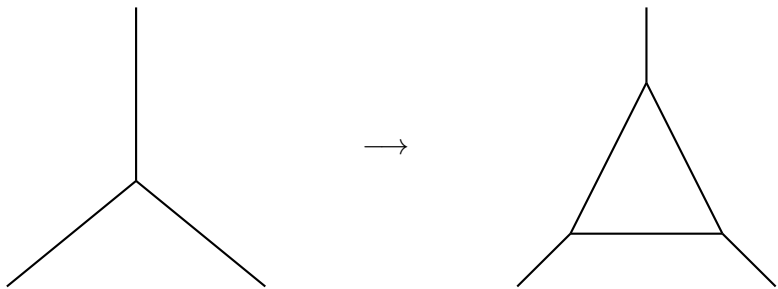
\begin{figure}
    \centering
 \begin{minipage}{.45\textwidth}
  \centering
\begin{tikzpicture}[scale=0.5]
 \draw [thick](-1.4,-1.4)-- (2,1.4);
 \draw [thick] (2,1.4)-- (2,6);
 \draw [thick] (2,1.4)-- (5.4,-1.4);
\end{tikzpicture}

 \end{minipage}
 $\longrightarrow$
 \begin{minipage}{.45\textwidth}
 \centering    
\begin{tikzpicture}[scale=0.5]

\draw [white] (-2,0)--(-1,0);
 \draw [thick](0,0)-- (2,4);
 \draw [thick] (4,0)-- (2,4);
 \draw [thick] (0,0)-- (4,0);
 \draw [thick] (2,4)--(2,6);
 \draw [thick] (-1.4,-1.4)--(0,0);
 \draw [thick] (4,0)--(5.4,-1.4);
\end{tikzpicture}
  
 \end{minipage}

\caption{Truncating a vertex}\label{fig:trunc}
\end{figure}

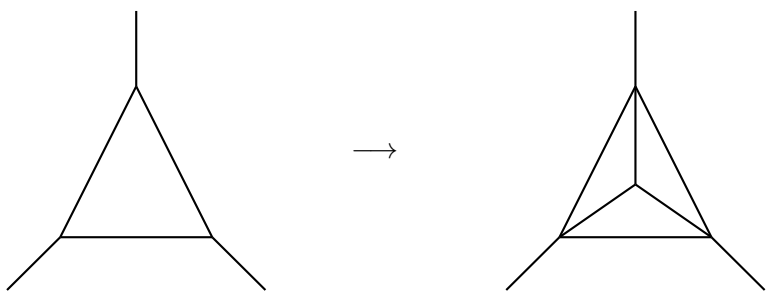
\begin{figure}
    \centering
      \begin{minipage}{.45\textwidth}
 \centering    
\begin{tikzpicture}[scale=0.5]

\draw [white] (-2,0)--(-1,0);
 \draw [thick](0,0)-- (2,4);
 \draw [thick] (4,0)-- (2,4);
 \draw [thick] (0,0)-- (4,0);
 \draw [thick] (2,4)--(2,6);
 \draw [thick] (-1.4,-1.4)--(0,0);
 \draw [thick] (4,0)--(5.4,-1.4);
\end{tikzpicture}
   \end{minipage}
 $\longrightarrow$
  \begin{minipage}{.45\textwidth}
 \centering    
\begin{tikzpicture}[scale=0.5]

\draw [white] (-2,0)--(-1,0);
 \draw [thick](0,0)-- (2,4);
 \draw [thick] (4,0)-- (2,4);
 \draw [thick] (0,0)-- (4,0);
 \draw [thick] (2,4)--(2,6);
 \draw [thick] (-1.4,-1.4)--(0,0);
 \draw [thick] (4,0)--(5.4,-1.4);
 
 \draw [thick](0,0)-- (2,1.4);
 \draw [thick] (2,1.4)-- (2,4);
 \draw [thick] (2,1.4)-- (4,0);
\end{tikzpicture}
   \end{minipage}
\caption{Triangulating a face}\label{fig:triang}
\end{figure}

The Upper Bound Conjecture naturally leads to the question of what is the supremum of all volumes of polyhedra sharing the same $1$-skeleton. This is answered in \cite[Theorem 4.2]{maxvol} by the following

\begin{teo}
 For any polyhedral graph $\Gamma$,
 $$\sup_{P}\mathrm{Vol}(P)=\vol(\overline{\Gamma})$$
 where $P$ varies among all proper generalized hyperbolic polyhedra with $1$-skeleton $\Gamma$ and $\overline{\Gamma}$ is the rectification of $\Gamma$.
\end{teo}

The rectification of a graph is defined in \cite[Section 3.4]{maxvol} (see also Remark \ref{rmk:rect}); for the purpose of this paper it suffices to say that $\overline{\Gamma}$ is the polyhedron with $1$-skeleton $\Gamma$ with every edge tangent to $\partial\h$ in the Klein model of hyperbolic space (hence, which has dihedral angle $0$ at each edge). This polyhedron can be canonically truncated to give an ideal right-angled hyperbolic polyhedron, hence it makes sense to speak of $\vol(\overline{\Gamma})$ as the volume of the truncation. 

As an application of Theorem \ref{teovol}, we prove in Theorem \ref{teo:volconj} that the Turaev-Viro Volume Conjecture holds for a new infinite family of cusped manifolds. These are complements of certain links in $S^3\#^g(S^1\times S^2)$; their hyperbolic structure is obtained by gluing right-angled octahedra.

In Section \ref{sec:yok} we set the notation, give the basic properties of the Kauffman bracket and define the Yokota invariant. In Section \ref{sec:volconj} we discuss previous Volume Conjectures for polyhedra and state the Upper Bound Conjecture. In Section \ref{sec:fourier} we introduce the Fourier transform of Barrett, and use it to prove Theorem \ref{teovol}. Section \ref{sec:tvvolconj} contains the proof of the Turaev-Viro Volume Conjecture for a new family of manifolds. Finally in an appendix we propose numerical evidence for a related Volume Conjecture for polyhedra.

\textbf{Acknowledgments.} I wish to thank my advisors Bruno Martelli and Francesco Costantino for their guidance and support. Furthermore I wish to thank Renaud Detcherry, Efstratia Kalfagianni and Tian Yang for their comments on a preliminary version of this paper. Finally I would like to thank the anonymous referee for their detailed suggestions which greatly improved the readability of this paper. Part of this work supported by the Deutsche Forschungsgemeinschaft under Germany’s Excellence Strategy EXC-2181/1 -390900948 (the Heidelberg STRUCTURES Cluster of Excellence).

\section{The Kauffman bracket and the Yokota invariant}\label{sec:yok}

\subsection{The Kauffman bracket}
Throughout
the rest of the paper $r\geq 3$ is an odd integer and $q=e^{\frac{2\pi i}{r}}$. All the definitions we give in this section are standard; the only notable difference is that in some papers (e.g. \cite{bar}) the graphs are colored with half-integer colors, while here we use integers.

Given a non-negative integer $n$, the \emph{quantum integer}
$[n]$ is defined as $\frac{q^n-q^{-n}}{q-q^{-1}}=\frac{\sin\left(2\pi n/r\right)}{\sin\left(2\pi/r\right)}$, and the quantum factorial $[n]!$ is $\prod_{i=1}^n[i]$ (with the convention that $[0]!=1$). Furthermore,
we denote with $I_r$ the set of all even non-negative integers at most equal to $r-2$.

\begin{oss}
 Because of the choice of root of unity $q$, we need to work with the $SO(3)$ version of the quantum invariants, rather than the $SU(2)$ version. This essentially amounts to using only even numbers as colors; a brief overview of how these invariants are related can be found for example in Section 2 of \cite{DKY}. Because of this, some terms in the upcoming formulas appear redundant; we still include them to keep the notation uniform with other papers dealing with the $SU(2)$ version.
\end{oss}

\begin{dfn}\label{dfn:radm}
 We say that a triple $(a,b,c)\in I_r^3$ is \emph{$r$-admissible} if
 \begin{itemize}
  \item $a,b,c\leq r-2$;
  \item $a+b+c$ is even and $a+b+c\leq 2r-4$;
  \item $a\leq b+c$, $b\leq a+c$ and $c\leq a+b$. 
 \end{itemize}
We say that a $6$-tuple $(n_1,n_2,n_3,n_4,n_5,n_6)$ of elements in $I_r$ is \emph{$r$-admissible} if the $4$ triples $(n_1,n_2,n_3)$,
$(n_1,n_5,n_6)$, $(n_2,n_4,n_6)$ and $(n_3,n_4,n_5)$ are $r$-admissible.
\end{dfn}

For $n\in \mathbb{N}$ define
\begin{equation}
 \Delta_n=(-1)^{n+1}[n+1].
\end{equation}

For an $r$-admissible triple $(a,b,c)$ we can define
\begin{equation}
 \Theta(a,b,c)=(-1)^{\frac{a+b+c}{2}}\frac{[\frac{a+b+c}{2}+1]!}{[\frac{a+b-c}{2}]![\frac{a-b+c}{2}]![\frac{-a+b+c}{2}]!}
\end{equation}

and $\Delta(a,b,c):=\Theta(a,b,c)^{-\frac{1}{2}}$.
 Notice that the number inside the square root is real; by convention we take the positive square root of a positive number, and the square root with positive imaginary part of a negative number.

\begin{dfn}\label{dfn:adm}

 An \emph{$r$-admissible coloring} for a tetrahedron $T$ is an assignment of an $r$-admissible $6$-tuple $(n_1,n_2,n_3,n_4,n_5,n_6)\in I_r^6$ to the set of edges of $T$, as
 shown in Figure \ref{fig:colamm}. More generally, we say that an $r$-admissible coloring for a trivalent graph $\Gamma\subseteq S^3$ is an assignment of elements of $I_r$ to the edges of $\Gamma$ such that the colors at each vertex form an $r$-admissible triple. Even more generally we say that an assignment of elements of $I_r$ to edges of a (not necessarily trivalent)  graph is a \emph{coloring}, and a graph $\Gamma$ together with its coloring $col$ is a \emph{colored graph} $(\Gamma,col)$.
\end{dfn}
 
If $v$ is a trivalent vertex of a graph whose incident edges are colored by an $r$-admissible triple $a,b,c$ we write for short $\Theta(v)$ and $\Delta(v)$ instead of $\Theta(a,b,c)$ and $\Delta(a,b,c)$.

Moreover, for an $r$-admissible $6$-tuple $(n_1,n_2,n_3,n_4,n_5,n_6)$ we can define its $6j$-symbol as usual as

\begin{gather}
 \begin{vmatrix}
   n_1 &n_2&n_3\\
   n_4&n_5&n_6
  \end{vmatrix}=\prod_{i=1}^4\Delta(v_i)
 \sum_{z=\max T_i}^{\min Q_j}\frac{(-1)^z[z+1]!}{\prod_{i=1}^4[z-T_i]!\prod_{j=1}^3[Q_j-z]!}\label{sixj}
\end{gather}
where:
\begin{itemize}
 \item $v_1=(n_1,n_2,n_3)$, $v_2=(n_1,n_5,n_6)$, $v_3=(n_2,n_4,n_6)$, $v_4=(n_3,n_4,n_5)$;
 \item $T_1=\frac{n_1+n_2+n_3}{2}$, $T_2=\frac{n_1+n_5+n_6}{2}$, $T_3=\frac{n_2+n_4+n_6}{2}$ and $T_4=\frac{n_3+n_4+n_5}{2}$;
 \item $Q_1=\frac{n_1+n_2+n_4+n_5}{2}$, $Q_2=\frac{n_1+n_3+n_4+n_6}{2}$ and $Q_3=\frac{n_2+n_3+n_5+n_6}{2}$.
\end{itemize}

By convention we define the $6j$-symbol of a non $r$-admissible tuple to be equal to $0$.
\iffalse\begin{oss}
 Notice that if $z\geq r-1$ the summand in \eqref{sixj} corresponding to $z$ is equal to $0$.
\end{oss}
\fi

\begin{figure}
\centering
 \begin{tikzpicture}
 \draw [thick] (0,0)-- node[below] {$n_4$}(4,0);
 \draw [thick] (0,0)-- node[above, left] {$n_6$}(2,4);
 \draw [thick] (2,4)-- node[above,right] {$n_5$}(4,0);
 \draw [thick] (0,0)-- node[below,right] {$n_2$}(2,1.4);
 \draw [thick] (2,1.4)-- node[left] {$n_1$} (2,4);
 \draw [thick] (2,1.4)-- node [below,left] {$n_3$}(4,0);
\end{tikzpicture}
\caption{An $r$-admissible coloring for a tetrahedron}\label{fig:colamm}
\end{figure}
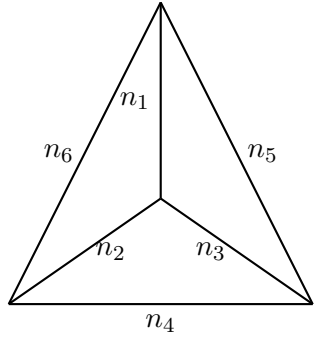

\begin{figure}
\centering
 \begin{tikzpicture}
 \draw [thick] (0,0) circle (1cm);
 \draw [thick] (0,1)--(0,-1);
\end{tikzpicture}
\caption{A Theta graph}\label{fig:theta}
\end{figure}
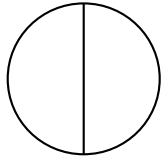
 
 The \emph{Kauffman bracket} is an invariant of \emph{trivalent framed graphs}; before defining the Kauffman bracket we remind the definition of framed graphs.
 
\begin{dfn}
 A \emph{framed graph} $\Gamma\subseteq S^3$ is a graph in $S^3$ together with a $2$-dimensional oriented thickening, considered up to isotopy. More precisely, a framed graph $\Gamma$ is a pair $G,F$ with $G$ an embedded graph in $S^3$ and $F$ an embedded orientable surface containing $G$ as a deformation retract. As is usual for framed links, we draw planar diagrams of framed graphs with over and under crossing information, and such that the ``thickness'' of the surface always lies flat on the projection plane.
\end{dfn}

\begin{dfn}\label{def:kauf}
 The \emph{Kauffman bracket} is the unique map $$\langle\cdot\rangle:\{\textrm{colored trivalent framed graphs in } S^3\}\ra \C$$ satisfying the following properties:
 \begin{enumerate}[(i)]
 \item\label{prop:prima} If $\Gamma$ is the planar circle colored with $n\in I_r$ then $\langle\Gamma\rangle=\Delta_n$;
  \item If $\Gamma$ is a Theta graph (see Figure \ref{fig:theta}) colored with the $r$-admissible triple $\left(a,b,c\right)\in I_r^3$ then $\langle\Gamma\rangle=1$;
  \item If $\Gamma$ is a tetrahedron colored with the $r$-admissible $6$-tuple $(n_1,\dots,n_6)\in I_r^6$ then $$\langle\Gamma\rangle=\begin{vmatrix}
   n_1 &n_2&n_3\\
   n_4&n_5&n_6
  \end{vmatrix};$$
  \item\label{prop:fusion} The \emph{fusion rule}:
\begin{equation}\label{eq:fusion} \left\langle\vcenter{\hbox{\includegraphics[width=2cm]{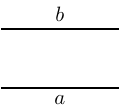}}}\right\rangle=\sum_{i\in I_r}\Delta_i\left\langle\vcenter{\hbox{\includegraphics[width=3cm]{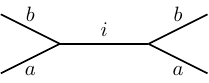}}}\right\rangle\end{equation}
  \item \label{prop:bridge} If $\Gamma$ has a bridge (that is to say, an edge that disconnects the graph if removed) colored with $i\neq 0$, then $\langle\Gamma\rangle=0$;
  \item If at some vertex of $\Gamma$ the colors do not form an $r$-admissible triple, then $\langle\Gamma\rangle=0$;
  \item \label{prop:zero}If $\Gamma$ is colored with an $r$-admissible coloring such that the color of an edge $e$ is equal to $0$, then \begin{equation} \left\langle\vcenter{\hbox{\includegraphics[width=2cm]{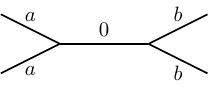}}}\right\rangle=\frac{1}{\sqrt{\Delta_a\Delta_b}}\left\langle\vcenter{\hbox{\includegraphics[height=0.8cm]{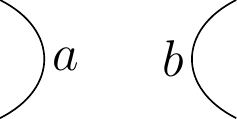}}}\right\rangle\end{equation}
  
  $\langle\Gamma\rangle=\frac{1}{\sqrt{\Delta_a\Delta_b}}\langle\Gamma'\rangle$ where $\Gamma'$ is $\Gamma$ with $e$ removed, and $a,b$ are the colors of the edges that share a vertex with $e$ (notice that since the coloring is $r$-admissible, two edges sharing the same vertex with $e$ will have the same color);
  \item \label{prop:crossing} The undoing of a crossing:
  \begin{displaymath}
   \left\langle\vcenter{\hbox{\includegraphics[height=2cm]{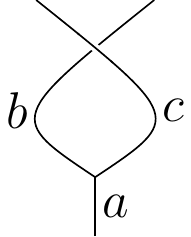}}}\right\rangle=(-1)^{\frac{b+c-a}{2}}q^{\frac{b(b+2)+c(c+2)-a(a+2)}{4}}\left\langle\vcenter{\hbox{\includegraphics{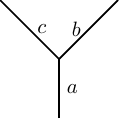}}}\right\rangle
  \end{displaymath}

  \item \label{prop:ultima} If $\Gamma$ is the disjoint union of $\Gamma_1$ and $\Gamma_2$, then $\langle\Gamma\rangle=\langle\Gamma_1\rangle\langle\Gamma_2\rangle$.
 \end{enumerate}

\end{dfn}

It is absolutely not clear from the definition that such a map exists; a proof is in \cite[Chapter 9]{kauflins}. However, it is straightforward to see that Properties \ref{prop:prima}-\ref{prop:ultima} are enough to calculate $\langle\Gamma\rangle$. Taking any planar diagram of $\Gamma$, apply a fusion rule near each crossing, and then undo the crossing using Property \ref{prop:crossing}; therefore we only need to calculate $\langle\cdot\rangle$ on planar graphs. For a planar graph, repeated applications of the Fusion rule create a bridge, and Properties \ref{prop:bridge}, \ref{prop:zero} and \ref{prop:ultima} allow to compute $\langle\Gamma\rangle$ from the Kauffman bracket of two graphs, each with fewer vertices.

\begin{oss}
 There are a few different normalizations of the Kauffman bracket in the literature. Here we use the \emph{unitary normalization}; it should be noted that \cite{kauflins} uses a different one, however the results there apply to the unitary normalization with little modification.
\end{oss}

In what follows sometimes we will color the edges of $\Gamma$ with linear combinations of colors; the Kauffman bracket can be extended linearly to this context. In particular, we will use \emph{Kirby's color} $\Omega:=\sum_{i\in I_r}\Delta_i i $.

\subsection{The definition of the Yokota invariant from the Kauffman bracket}

In this subsection we give an overview of the Yokota invariant, which generalizes the Kauffman bracket invariant of trivalent graphs to graphs with vertices of any valence; it was first introduced in \cite{yok}.

Suppose $\Gamma\subseteq S^3$ is a framed graph with vertices of valence at least $3$; as before $r>2$ is odd and $q=e^{2\pi i/r}$.

For a vertex $v$ of $\Gamma$, we can take a small ball $B$ containing $v$, and replace $\Gamma\cap B$ with a trivalent planar tree in $B$ having the same endpoints in $\partial B\cap \Gamma$ (see Figure \ref{fig:desing}). We call this procedure a \emph{desingularization} of $\Gamma$ at $v$. Notice that if $v$ has valence greater than $3$, then this procedure is not unique; however, any desingularization is related to any other via a sequence of Whitehead moves (see Figure \ref{fig:whitehead}). This fact can be most easily seen by thinking about the dual graph: the vertex corresponds to a polygon and a desingularization corresponds to a choice of enough diagonals to triangulate the polygon. Then a Whitehead move acts on the dual as a diagonal flip, and clearly diagonal flips are enough to go from any choice of diagonals to any other.

\begin{figure}
\centering
 \begin{minipage}{.35\textwidth} \begin{tikzpicture}[scale=0.8]
\centering

\draw[thick] (3,3)--(3,-3);
\draw[thick] (0.5,2)--(5.5,-2);
\draw[thick] (5.5,2)--(0.5,-2);
\end{tikzpicture}
\end{minipage}
$\xrightarrow{\hspace*{1cm}}$
\begin{minipage}{.3\textwidth}
  \begin{tikzpicture}[scale=0.6]
\centering

\draw[thick] (3,3)--(3,0);
\draw[thick] (3,0)--(3,-4);
\draw[thick] (0.5,2)--(3,0);
\draw[thick] (5.5,2)--(3,-1);
\draw[thick](0.5,-4)--(3,-1.5);
\draw[thick] (3,-2)--(5.5,-4);
\end{tikzpicture}
\end{minipage}
\caption{Desingularization of a vertex of valence $6$}\label{fig:desing}
\end{figure}
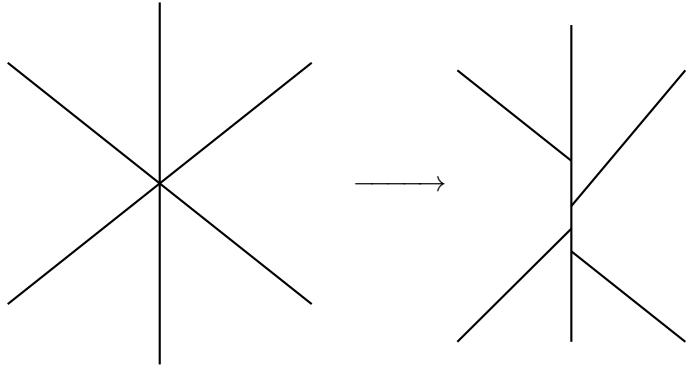

We say that the trivalent graph $\Gamma'$ is a desingularization of $\Gamma$ if it is obtained from $\Gamma$ by desingularization of each vertex of valence at least $4$.

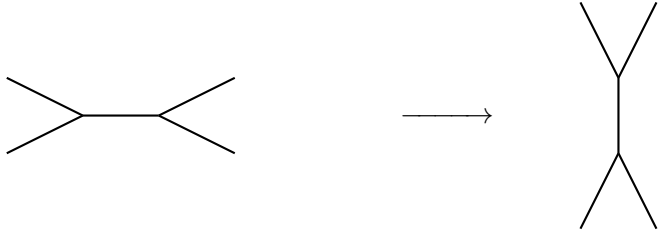
\begin{figure}
 \centering
 \begin{minipage}{.4\textwidth}
    \begin{tikzpicture}
\centering
\draw[thick](-1,0.5)--(0,0);
\draw[thick](-1,-0.5)--(0,0);
\draw[thick] (0,0)--(1,0);
\draw[thick](1,0)--(2,0.5);
\draw[thick](1,0)--(2,-0.5);
\end{tikzpicture}
 \end{minipage}
$\xrightarrow{\hspace*{1.0cm}}$\hspace*{1cm}
\begin{minipage}{.2\textwidth}
  \begin{tikzpicture}[rotate=90]
\centering
\draw[thick](-1,0.5)--(0,0);
\draw[thick](-1,-0.5)--(0,0);
\draw[thick] (0,0)--(1,0);
\draw[thick](1,0)--(2,0.5);
\draw[thick](1,0)--(2,-0.5);
\end{tikzpicture}
 \end{minipage}
 \caption{A Whitehead move}\label{fig:whitehead}
\end{figure}

\begin{dfn}\label{dfn:Yok}
Let $(\Gamma,col)$ be a framed graph in $S^3$ colored with elements of $I_r$. Let $\Gamma'$ be a desingularization of $\Gamma$. Call $e_1',\dots,e_k'$ the edges of $\Gamma'$ that were added by the desingularization. 
 If $k>0$, then the \emph{Yokota invariant} of $(\Gamma,col)$ is 
 \begin{displaymath}
  Y_r(\Gamma,col):=\sum_{col'\in I_r^k}\left(\prod_{i=1}^k\Delta_{col'(e_i')}\right)|\langle\Gamma',col\cup col'\rangle|^2
 \end{displaymath}
with $col'$ coloring the edges $e_1',\dots,e_k'$. If instead $k=0$ (i.e. $\Gamma=\Gamma'$, i.e. $\Gamma$ is trivalent) then $Y_r(\Gamma,col)=|\langle \Gamma,col\rangle|^2$.
\end{dfn}

As we did with the Kauffman bracket, we extend linearly the Yokota invariant to linear combinations of colors. Notice that in this case, even if $\Gamma$ is trivalent, we may get $Y_r(\Gamma,col)\neq \lvert\langle \Gamma,col\rangle\rvert^2$.

\begin{oss}
 We stress the fact that we are using the unitary normalization for the Kauffman bracket. If we instead used the Kauffman normalization $\langle\cdot\rangle_K$ of \cite{kauflins}, the definition of the Yokota invariant of $\left(\Gamma,col\right)$ would be
 
  \begin{displaymath}
  Y_r(\Gamma,col):=\sum_{col'\in I_r^k}\frac{\prod_{i=1}^k\Delta_{col'(e'_i)}}{\prod_{v \textrm{ vertex of }\Gamma}\Theta(v)}|\langle\Gamma',col\cup col'\rangle_K|^2
 \end{displaymath}
\end{oss}

\begin{prop}\cite{yok}
 The Yokota invariant does not depend on the choice of desingularization.
\end{prop}

We can easily extend the Yokota invariant to graphs with $1$-valent and $2$-valent vertices as well via the following formulas.

$$ Y_r\left(\includegraphics[width=2cm]{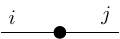}\right)=\frac{\delta_{i,j}}{\Delta_i}Y_r\left(\includegraphics[width=3cm]{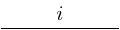}\right);$$

$$ Y_r\left(\vcenter{\hbox{\includegraphics{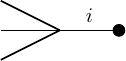}}}\right)=\delta_{i,0}Y_r\left(\vcenter{\hbox{\includegraphics{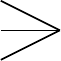}}}\right);$$

We normalize the invariant so that it is equal to $1$ for the graph with a single vertex and no edges.

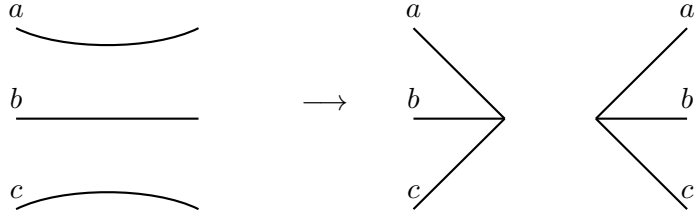
\begin{figure}
 \centering
 \begin{minipage}{.4\textwidth}
 \centering
  \begin{tikzpicture}[scale=0.6]
   \draw[thick] (0,4) node [above]{$a$}--(2,2);
   \draw[thick] (0,2) node [above]{$b$}--(2,2);
   \draw[thick] (0,0) node [above]{$c$}--(2,2);
   \draw[thick] (4,2)--(6,4) node [above]{$a$};
   \draw[thick] (4,2)--(6,2) node [above]{$b$};
   \draw[thick] (4,2)--(6,0) node [above]{$c$};
  \end{tikzpicture}
 \end{minipage}
 $\longrightarrow$ 
\begin{minipage}{.4\textwidth}
\centering
\begin{tikzpicture}[scale=0.6]
 \draw[thick] (0,4) node [above]{$a$} .. controls (1,3.5) and (3,3.5) .. (4,4);
 \draw[thick] (0,2) node [above]{$b$} --(4,2);
 \draw[thick] (0,0) node [above]{$c$}..controls (1,0.5) and (3,0.5)..(4,0);
 \end{tikzpicture}
\end{minipage}
\caption{A vertex sum of two trivalent vertices}\label{fig:vertsum}
\end{figure}

Now we give three important properties of the Yokota invariant, all easy consequences of the definitions.

\begin{prop}\label{prop:yokota}
The following hold:
 \begin{enumerate}
 \item\label{prop:framing} The Yokota invariant does not depend on the framing of $\Gamma$.
  \item\label{prop:whit} If an edge $e$ of $\Gamma$ is colored with the Kirby color $\Omega$, and $\Gamma'$ is obtained from $\Gamma$ via a Whitehead move on the edge $e$ (coloring the edge that replaces $e$ with $\Omega$ and keeping every other color the same) then $Y_r(\Gamma,col)=Y_r(\Gamma',col)$.
  \item\label{prop:vertexsum} If $\Gamma$ is a vertex sum of $\Gamma_1,\Gamma_2$ along trivalent vertices $v_1\in\Gamma_1$ and $v_2\in\Gamma_2$ (see Figure \ref{fig:vertsum}), then $Y_r(\Gamma,col)=Y_r(\Gamma_1,col_1)Y_r(\Gamma_2,col_2)$ where $col_1,col_2$ are the restrictions of $col$ to $\Gamma_1,\Gamma_2$ respectively.
 \end{enumerate}

\end{prop}
\begin{proof}
Part \ref{prop:framing} holds because $\langle\Gamma\rangle$ depends on the framing of $\Gamma$ only up to a factor of $q^a$, thus when taking squared norms this becomes $1$.
 Part \ref{prop:whit} is essentialy the fact that the Yokota invariant is well defined: since both $e$ and the corresponding edge in $\Gamma'$ are colored with $\Omega$, both sides of the equality are equal to the Yokota invariant of the graph obtained by collapsing $e$ to a point. 
 
 Part \ref{prop:vertexsum} follows from the analogous property for the Kauffman bracket; this is obtained via two applications of the fusion rule and one application of the Bridge rule of Definition \ref{def:kauf}.\ref{prop:bridge} (see Figure \ref{fig:vertexsum}).
\end{proof}

\begin{figure}
\centering
 \begin{minipage}{.3\textwidth} \begin{tikzpicture}[scale=0.6]
\centering

\draw[thick] (0,3)--(6,3);
\draw[thick] (0,1.5)--(6,1.5);
\draw[thick] (0,0)--(6,0);
\end{tikzpicture}
\end{minipage}
$\xrightarrow{\hspace*{1.5cm}}$
\begin{minipage}{.3\textwidth}

  \begin{tikzpicture}[scale=0.6]
\centering
\draw[thick] (0,3)--(2,1);
\draw[thick] (0,1.5)--(1,2);
\draw[thick] (0,0)--(2,1);
\draw[thick] (2,1)--(4,1);
\draw[thick] (4,1)--(6,0);
\draw[thick](4,1)--(6,3);
\draw[thick] (6,1.5)--(5,2);

\end{tikzpicture}

\end{minipage}
\caption{Applying the fusion rule to three edges arising from a vertex sum.}\label{fig:vertexsum}
\end{figure}
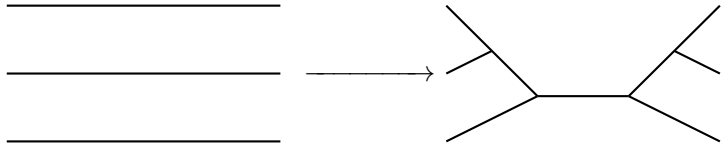

 It is very important that the vertex sum in Proposition \ref{prop:yokota}.\ref{prop:vertexsum} is done between trivalent vertices; the assertion is false in general.

\begin{oss}
 The Kauffman bracket (hence, the Yokota invariant) can also be defined in the much larger setting of framed trivalent graphs in closed oriented $3$-manifolds (see for example \cite{lickorish} and \cite{kauflins}); since we will not need such a generality that carries some more technical details, we will restrict ourselves to the $S^3$ case.
\end{oss}

\section{Volume Conjecture for polyhedra}\label{sec:volconj}

\subsection{The Volume Conjecture for polyhedra}
Costantino first conjectured in \cite{C} that the growth of the $6j$-symbol is given by the volume of a hyperbolic tetrahedron. A Volume Conjecture for trivalent graphs (and their Kauffman bracket invariant) was proposed in \cite{vdv} and later refined in \cite{volconjpoly} to the case of planar trivalent graphs and hyperbolic polyhedra with trivalent vertices. The conjecture of \cite{volconjpoly} evaluates the invariant at the first root of unity $q=e^{\pi i/r}$; the downside of this choice is that they have to consider poles of the Kauffman bracket, instead of its values directly. Recently, Murakami and Kolpakov \cite{murkolp} proposed a Volume Conjecture for polyhedra at the second root of unity $q=e^{2\pi i/r}$, but only stated it for simple polyhedra without hyperideal vertices (see Remark \ref{rmk:simple} and Definition \ref{dfn:poly}); remarkably this conjecture directly involves the value of the Kauffman bracket. Here we propose Conjecture \ref{conjmk} which is an extension of Kolpakov-Murakami's Volume Conjecture to a very general setting, and then propose Conjecture \ref{cng:boundgen} which concerns an upper bound for the Yokota invariant of polyhedral graphs.

\textit{Geometric background.}

 Recall the projective model for hyperbolic space $\mathbb{H}^3\subseteq\mathbb{R}^3\subseteq\mathbb{RP}^3$ where $\mathbb{H}^3$ is the unit ball of $\R^3$ (for the basic definitions see for example \cite{bonbao}). Notice that for convenience we have picked an affine chart $\mathbb{R}^3\subseteq\mathbb{RP}^3$, so that it always make sense to speak of segments between two points, half spaces, etcetera; this choice is inconsequential, up to isometry. It should be mentioned that isometries, in this model, correspond to projective transformations that preserve the unit sphere.
 
The space $\mathbb{RP}^3$ has a duality that comes from the Minkowski scalar product on $\R^{3,1}$; using this we can associate to a point $p$ lying in $\mathbb{R}^3\backslash \overline{\mathbb{H}^3}$ a plane $\Pi_p\subseteq \h$, called the \emph{polar plane} of $p$, such that all lines passing through $\mathbb{H}^3$ and $p$ are orthogonal to $\Pi_p$ (see Figure \ref{fig:dual} for a $2$-dimensional picture).
If $p\in \mathbb{R}^3\backslash\overline{\mathbb{H}^3}$, denote with $H_p\subseteq\mathbb{H}^3$ the half space delimited by the polar plane $\Pi_p$ on the other side of $p$; in other words, $H_p$ contains $0\in\R^3$. If the line from $p$ to $p'$ passes through $\mathbb{H}^3$, then $\Pi_p$ and $\Pi_{p'}$ are disjoint \cite[Lemma 4]{bonbao}. In particular, if the segment from $p$ to $p'$ intersects $\mathbb{H}^3$, then $\Pi_p\subseteq H_{p'}$ and $\Pi_{p'}\subseteq H_{p}$; if however the segment does not intersect $\h$, but the half line from $p$ to $p'$ does, then $H_{p}\subseteq H_{p'}$. If $p$ gets pushed away from $\h$, then $\Pi_p$ gets pushed closer to the origin of $\R^3$. 

\begin{figure}
 \includegraphics[scale=0.9]{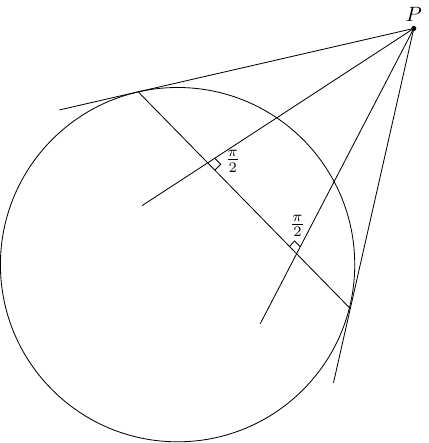}\caption{The dual of a point $P$}\label{fig:dual}
\end{figure}
 
 \begin{dfn}
  A \emph{projective polyhedron} in $\mathbb{RP}^3$ is a convex polyhedron in some affine chart of $\mathbb{RP}^3$. Alternatively, it is the closure of a connected component of the complement of finitely many planes in $\mathbb{RP}^3$ that does not contain any projective line.
 \end{dfn}

\begin{dfn} Following \cite[Definition 4.7]{rivhodg}\label{dfn:poly}
 \begin{itemize}
  \item 
We say that a projective polyhedron $P\subseteq \R^3\subseteq\mathbb{RP}^3$ is a \emph{generalized hyperbolic polyhedron} if each edge of $P$ intersects $\mathbb{H}^3$.
\item A vertex of a generalized hyperbolic polyhedron is \emph{real} if it lies in $\h$, \emph{ideal} if it lies in $\partial \h$ and \emph{hyperideal} otherwise.
\item a generalized hyperbolic polyhedron $P$ is \emph{proper} if for each hyperideal vertex $v$ of $P$ \emph{the interior} of the polar half space $H_v$ contains all the other real vertices of $P$ (see Figure \ref{fig:propbadtrunc}). 
\item 
 We define the \emph{truncation} of a generalized hyperbolic polyhedron $P$ at a hyperideal vertex $v$ to be the intersection of $P$ with $H_v$; similarly the \emph{truncation} of $P$ is the truncation at every hyperideal vertex, that is to say $P\cap\left(\cap_{v \textrm{ hyperideal}}H_v\right)$. We say that the \emph{volume} of $P$ is the volume of its truncation. Notice that the volume of a non-empty generalized hyperbolic polyhedron could be $0$ if the truncation is empty.
 \end{itemize}

\end{dfn}

\iffalse
 Up to a projective transformation preserving the unit sphere (which is the same as an isometry of $\h$) any generalized hyperbolic polyhedron is contained in the standard affine chart $\R^3\subseteq \h$; for convenience we are going to assume that this is always the case.\fi

\begin{figure}
  \centering
    \begin{tikzpicture}
\centering
\fill[fill=gray](2.5,2.5)--(3.5,2)--(2.9,1.97)node[above]{$\Pi_v$};
\draw[thick] (1.5,1.5)--(4,4)node[above]{$v$};
\draw[thick](3.25,1)--(4,4);
\draw[thick] (2.5,1.25)--(4,4);
\end{tikzpicture}
\caption{A proper vertex}\label{fig:propbadtrunc}
\end{figure}

In the remainder of the paper we simply say \emph{proper polyhedra} for proper generalized hyperbolic polyhedra.

When it has positive volume, the truncation of a generalized hyperbolic polyhedron $P$ is itself a polyhedron; some of its faces are the truncation of the faces of $P$, while the others are the intersection of $P$ with some truncating plane; we call such faces \emph{truncation faces}. If an edge of the truncation of $P$ lies in a truncation face we say that the edge is arising from the truncation. 

\begin{oss}
 For proper polyhedra the dihedral angles at the edges arising from the truncation are $\frac{\pi}{2}$.
\end{oss}
\iffalse
\begin{oss}
 An important feature of the truncation of a proper polyhedron $P$ is that it determines $P$ (once we know which faces of $P$ are truncation faces), since it is enough to remove the truncation faces to undo the truncation (see Figure \ref{fig:recover}). 
\end{oss}

  \begin{figure}
  \centering
\begin{minipage}{.45\textwidth}
\centering
 \begin{tikzpicture}[scale=0.7]
\draw [thick,dashed](5,5)--(4,0);
\draw [thick, dashed](5,5)--(1.5,0);
\draw[thick,dashed](5,5)--(0,1);
\draw [thick] (4,0)--(4.5,2.5);
\draw [thick] (0,1)--(3,3.4);
\draw [thick] (1.5,0)--(3,2.14);
\draw [thick] (4.5,2.5)--(3,3.4);
\draw [thick](3,2.14)- -(3,3.4);
\draw [thick] (3,2.14)--(4.5,2.5);
\end{tikzpicture}
\end{minipage}
$\longrightarrow$
\begin{minipage}{.45\textwidth}
\centering
 \begin{tikzpicture}[scale=0.7]
\draw [thick](5,5)--(4,0);
\draw [thick](5,5)--(1.5,0);
\draw[thick](5,5)--(0,1);
\draw [thick] (4,0)--(4.5,2.5);
\draw [thick] (0,1)--(3,3.4);
\draw [thick] (1.5,0)--(3,2.14);
\end{tikzpicture}
\end{minipage}
\caption{Removing the truncation faces recovers the original polyhedron.}\label{fig:recover}
 \end{figure}\fi
 
 \iffalse
We are always going to consider \emph{face marked} polyhedra; this means that each face of a polyhedron is uniquely determined, and therefore they never have any symmetry.\fi

\begin{oss}
 If $\Gamma$ can be embedded as the $1$-skeleton of a projective polyhedron, then it is $3$-connected (that is to say, it cannot be disconnected by removing two non-adjacent vertices). Furthermore, any $3$-connected planar graph can be embedded as the $1$-skeleton of a proper polyhedron \cite{steinitz}. 
 If a planar graph is $3$-connected, then it admits a unique embedding in $S^2$ (up to isotopies of $S^2$ and mirror symmetry) \cite[Corollary 3.4]{fle}. Hence when in the following we consider a planar graph, it is always going to be $3$-connected and embedded in $S^2$. In particular, it will make sense to talk about the dual of $\Gamma$, denoted with $\Gamma^*$. The graph $\Gamma^*$ is the $1$-skeleton of the cellular decomposition of $S^2$ dual to that of $\Gamma$.
\end{oss}

\iffalse
\begin{dfn}
 Let $\Gamma$ be a planar $3$-connected graph; the space of all the face-marked proper polyhedra with $1$-skeleton $\Gamma$ considered up to isometry (i.e. projective transformations preserving the unit sphere) is denoted as $\mathcal{A}_\Gamma$.
\end{dfn}
\fi

\begin{oss}
 It is important not to mix up the $1$-skeleton of a projective polyhedron with the $1$-skeleton of its truncation. In what follows, whenever we refer to
 $1$-skeletons we always refer to those of projective polyhedra (and not their truncation) unless specified.
\end{oss}

We propose the following formulation of the Volume Conjecture for polyhedra, generalizing the previously mentioned versions.

 \begin{cnj}[The Volume Conjecture for polyhedra]\label{conjmk}
Let $P$ be a proper polyhedron with dihedral angles
  $\alpha_1,\dots,\alpha_m$ at the edges $e_1,\dots,e_m$, and $1$-skeleton $\Gamma$. Let $col_r$ be a sequence of $r$-admissible colorings of the edges $e_1,\dots,e_m$ of $\Gamma$ such that 
  \begin{displaymath}
   2\pi\lim_{r\ra+\infty}\frac{col_r(e_i)}{r}=\pi-\alpha_i.
  \end{displaymath}
Then
\begin{displaymath}
 \lim_{r\ra+\infty}\frac{\pi}{r}\log\left\lvert Y_r(\Gamma,col_r)\right\rvert=\mathrm{Vol}(P).
\end{displaymath}

 \end{cnj}

 \begin{oss}\label{rmk:simple}
  In the case where $P$ is a simple polyhedron in $\mathbb{H}^3$ (i.e. a compact polyhedron with only trivalent vertices) this conjecture is the same as the Volume Conjecture of Kolpakov-Murakami \cite{murkolp}.
 \end{oss}

  Conjecture \ref{conjmk} was verified in \cite{chenmur} for tetrahedra with at least one hyperideal vertex; we provide some further supporting numerical evidence for Conjecture \ref{conjmk} for some pyramids in the Appendix \ref{appendix}, and prove it for a large family of examples in Proposition \ref{prop:r/2} and Remark \ref{rmk:volconj} (however, only for a single sequence of colors).
  
  \begin{oss}
   Conjecture \ref{conjmk} would imply that Conjecture \ref{cng:upperbound} is verified when restricted to colors which correspond to hyperbolic polyhedra. 
  \end{oss}

  \subsection{The upper bound conjecture}
  
  In \cite{bound6j} the authors proved an upper bound on the growth of the $6j$-symbol. When stated in terms of the Yokota invariant of the tetrahedral graph $T$, the result is the following.
  
\begin{teo}\label{teo:bound}
For any $r$ and any $r$-admissible coloring $col$ of the graph $T$, we have
$$ \frac{\pi}{r}\log \abs{Y_r(T,col)} \leq v_8+O\left(\frac{\log(r)}{r}\right).$$
  where $v_8\sim 3.66$ is the volume of the regular ideal right-angled octahedron.
  Furthermore, this inequality is sharp, with the upper bound achieved at the $6$-tuple $\frac{r-2\pm1}{2},\dots,\frac{r-2\pm1}{2}$ with the signs chosen so that $\frac{r-2\pm1}{2}$ is even.
\end{teo} 

It is natural to ask if a similar upper bound holds for other graphs. The reason the quantity $v_8$ is involved in the statement of Theorem \ref{teo:bound} is that it is the upper bound of the volume of all proper tetrahedra. In \cite{maxvol} the author proved that, given a polyhedral graph $\Gamma$, the upper bound of the volume of all proper polyhedra with $1$-skeleton $\Gamma$ is equal to the volume of the rectification of $\Gamma$, denoted with $\overline{\Gamma}$ (see Remark \ref{rmk:rect}). In light of this, we reword Conjecture \ref{cng:upperbound} as follows:

 \begin{cnj}\label{cng:boundgen}
  If $\Gamma$ is a polyhedral graph and $col$ is any $r$-admissible coloring of its edges, then
  \begin{displaymath}
   \frac{\pi}{r}\log\left\lvert Y_r(\Gamma,col)\right\rvert\leq \vol(\overline{\Gamma})+O\left(\frac{\log(r)}{r}\right)
  \end{displaymath}
 Moreover, the inequality is sharp, with equality attained by the sequence of colorings
giving the color $\frac{r-2\pm1}{2}$ to each edge (the sign is chosen so that the colors are even).
\end{cnj}

\begin{oss}\label{rmk:rect}
 The rectification of $\Gamma$ is defined as the unique projective polyhedron with $1$-skeleton $\Gamma$ and with every edge tangent to $\partial\h$ (see Figure \ref{fig:retttet}). While $\overline{\Gamma}$ is not a proper (or even generalized hyperbolic) polyhedron, we can still speak of its truncation and its volume; for more details see \cite[Section 3.4]{maxvol}.
\end{oss}

\begin{oss}
 It would be natural to ask whether a similar upper bound would work for non polyhedral graphs; however in this case it is unclear what would be the geometric object to replace $\overline{\Gamma}$.
\end{oss}

 \begin{figure}
  \centering
  \begin{minipage}{.4\textwidth}
   \centering
  \begin{tikzpicture}[rotate=20]
\centering
\draw (0,0) circle[radius=2cm];
\draw[thick] (2.2,2.2)--(1.84,-1.84)--(-2.2,-2.2)--(-1.84,1.84)--(2.2,2.2);
\draw[thick] (2.2,2.2)--(-2.2,-2.2);
\draw[thick,dashed](1.84,-1.84)--(-1.84,1.84);

\end{tikzpicture}
  \end{minipage}
  \hspace{1cm}
\begin{minipage}{.4\textwidth}
 \centering
   \begin{tikzpicture}[rotate=20]
\centering
\fill[color=gray](1.98,-0.15)--(-0.15,1.98)--(-0.1,-0.1)--(1.98,-0.15);
\fill[color=gray](-1.98,0.15)--(0.15,-1.98)--(-0.1,-0.1)--(-1.98,0.15);
\draw (0,0) circle[radius=2cm];
\draw[thick] (2.2,2.2)--(1.84,-1.84)--(-2.2,-2.2)--(-1.84,1.84)--(2.2,2.2);
\draw[thick] (2.2,2.2)--(-2.2,-2.2);
\draw[thick,dashed](1.84,-1.84)--(-1.84,1.84);
\draw[thin](1.98,-0.15)--(-0.15,1.98)--(-0.1,-0.1)--(1.98,-0.15);
\draw[thin](-1.98,0.15)--(0.15,-1.98)--(-0.1,-0.1)--(-1.98,0.15);
\draw[thin](-1.98,0.15)--(-0.15,1.98);
\draw[thin](1.98,-0.15)--(0.15,-1.98);
\draw[thin,dashed] (1.98,-0.15)--(0,0)--(0.15,-1.98);
\draw[thin,dashed] (-1.98,0.15)--(0,0)--(-0.15,1.98);
\end{tikzpicture}
\end{minipage}
\caption{The rectification of a tetrahedron (left) and its truncation (right), the ideal right-angled octahedron. The gray faces arise from the truncation of the top and bottom vertices.}\label{fig:retttet}
 \end{figure}

\begin{teo:maxvolconj}
 Conjecture \ref{cng:boundgen} is verified for any planar graph obtained from the tetrahedron by applying any sequence of the following two moves:
 \begin{itemize}
  \item blowing up a trivalent vertex (see Figure \ref{fig:trunc}) or
  \item triangulating a triangular face (see Figure \ref{fig:triang}).
 \end{itemize} 
\end{teo:maxvolconj}

This theorem will be proven in Section \ref{sec:fourier}.
 
\section{The Fourier Transform}\label{sec:fourier}
In this section we prove Theorem \ref{teovol}. The first main tool used is Theorem \ref{teo:bound}.

The second main tool used to prove Theorem \ref{teovol} is the Fourier Transform introduced in \cite{bar} by Barrett. We describe it here in a slightly different context and notation.

\begin{figure}
\centering
\begin{tikzpicture}
\begin{knot}[
    clip width=3,
    flip crossing={2},
    ]
    \strand [ultra thick] (0,0) circle (1.0cm);
    \strand [ultra thick] (1,0) circle (1.0cm);
\end{knot}
\end{tikzpicture} 
\caption{The $0$-framed Hopf link}\label{fig:Hopf}
\end{figure}
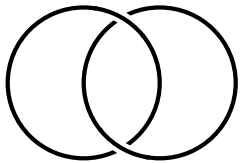

Let $H\subseteq S^3$ be the $0$-framed Hopf link as in Figure \ref{fig:Hopf}. For $i,j\in I_r$ we denote with $H(i,j)\in\C$ the value of the Kauffman bracket of the Hopf link colored with $i,j$; applying the relation of \cite[Figure 22]{lickorish} and an easy induction on $j$ shows that

\begin{gather*}
 H(i,j)=(-1)^{i+j}[(i+1)(j+1)]=(-1)^{i+j}\frac{\sin\left(\frac{2\pi}{r}\left(i+1\right)\left(j+1\right)\right)}{\sin\left(\frac{2\pi}{r}\right)}.
\end{gather*}

Furthermore denote with 

$$N:=\frac{r}{4\sin^2\left(\frac{2\pi}{r}\right)}=\langle U,\Omega\rangle=\sum_{i\in I_r} \Delta_i^2$$

where $U$ is the $0$-framed unknot in $S^3$ colored with the color $\Omega:=\sum\limits_{i\in I_r} \Delta_i i$ (see \cite[Page 185]{lickorish}).

\begin{oss}
 Once again we remark that we are using the $SO(3)$ version of the invariants evaluated at $q=e^{2\pi i/r}$. However, the Fourier transform and its properties hold with any choice of primitive $2r$-th root of unity, or any choice of primitive $4r$-th root of unity for the $SU(2)$ case; the proofs work verbatim in every other case.
\end{oss}

\begin{dfn}
 The \emph{Fourier transform} of $Y_r(\Gamma,col)$ is the invariant of the colored graph $(\Gamma,col')$ given by the formula
  \begin{displaymath}
\mathcal{F}_r(\Gamma,col')=\sum_{col \textrm{ coloring of }\Gamma}Y_r(\Gamma,col)H(col,col')
 \end{displaymath}
 where $$H(col,col'):=\prod_{e \textrm{ edge of }\Gamma}H(col(e),col'(e^*)).$$
\end{dfn}

The following proposition was first noticed by Barrett in \cite[Section 5]{bar}; a concise proof was later given in \cite[Theorem 1]{bar+}. For the sake of completeness, we include a detailed proof of this result.

\begin{prop}\label{prop:fourier}
 If $\Gamma$ is a planar framed graph, $\Gamma^*$ is its planar dual, and $col'$ is a coloring of the edges of $\Gamma^*$, then
 \begin{displaymath}
  Y_r(\Gamma^*,col')=N^{-g}\sum_{col \textrm{ coloring of }\Gamma}Y_r(\Gamma,col)H(col,col')=N^{-g} \mathcal{F}_r(\Gamma,col')
 \end{displaymath}
where $g$ is the genus of a regular neighborhood of $\Gamma$.
\end{prop}
\begin{proof}

The proof is entirely diagrammatic; when we display an equality between (linear combinations of) diagrams, we mean that they have the same Kauffman bracket. Throughout the proof we will liberally add $\Omega$-colored, $0$-framed unknots that are unlinked from anything else; this will generate an ambiguity of a power of $N$ that we will account for at the end.

\emph{Step 1:} Calculate $Y_r(\Gamma,col)$ as the Kauffman bracket of a certain framed colored link $L(\Gamma,col)$.

 The colored link $L(\Gamma,col)$ is obtained from $(\Gamma,col)$ as in Figure \ref{fig:chainmail}. Every vertex is replaced by a circle colored with $\Omega$, and every edge is replaced by a circle colored with the same color as the edge, wrapping around once each of the two circles corresponding to its vertices in a minimally twisted way (i.e. the circle has two consecutive overcrossings and two consecutive undercrossings). Notice that the link itself only depends on $\Gamma$; only its coloring depends on $col$. Furthermore we can define the framing to be the blackboard framing of the diagram we just constructed.
 
 \begin{figure}
 \centering
  \begin{minipage}{.4\textwidth}
     \begin{tikzpicture}[scale=0.4]
\centering
\draw[thick] (0,0) -- (10,0);
\draw[thick] (5,5)--(5,-5);
\draw[fill] (5,0) circle[radius=0.1cm];
\end{tikzpicture}
  \end{minipage}
$\xrightarrow{\hspace*{0.8cm}}$
  \begin{minipage}{.4\textwidth}
  \begin{tikzpicture}[scale=0.5]
\centering
\begin{knot}[
clip width=10,
]
\strand[thick] (0,0) .. controls (0.5,5) and (1.5,5) .. (2,0);
\strand[thick] (-4,4) .. controls(1,4.5) and (1,5.5) .. (-4,6);
\strand[thick] (6,4) .. controls(1,4.5) and (1,5.5) .. (6,6);
\strand[thick] (0,10) .. controls (0.5,5) and (1.5,5) .. (2,10);
\strand[thick] (1,5) circle[radius=2cm];
\flipcrossings{1,5,8,3}

\end{knot}
\node at (3,7) {\Large$\Omega$};
\end{tikzpicture}

  \end{minipage}
  \caption{Obtaining $L(\Gamma,col)$ using the \emph{Chainmail Rule.} Each circle has the same color as its corresponding edge, and it has two consecutive overcrossings and two consecutive undercrossings.}\label{fig:chainmail}
 \end{figure}

 The fact that $\langle L(\Gamma,col)\rangle=Y_r(\Gamma,col)$ can be shown by using the definition of $Y_r$ after applying the following identity to $L$:
 
 \begin{equation}\label{eqn:vertsum}
  \vcenter{\hbox{\includegraphics{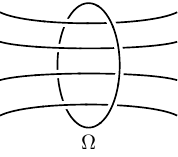}}}=\sum_{i\in I_r}\Delta_i \vcenter{\hbox{\includegraphics[height=3cm]{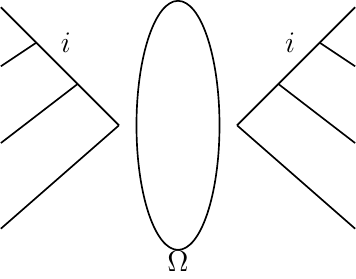}}}
 \end{equation}

 This holds for any number of strands; it is obtained by repeated application of the fusion rule followed by the well known fact (see \cite[Lemma 6]{lickorish}) that if a diagram contains the portion depicted in Figure \ref{fig:mapofoutside} it is equal to $0$ unless $i=0$.
 
 \begin{figure}
  \centering
  \begin{tikzpicture}
\centering
\begin{knot}[
clip width=5,
]
\strand[thick] (0,1) -- (3,1);
\strand[thick] (1.5,1.1) ellipse (15pt and 30pt);
\node at (1.5,-0.2) {$\Omega$};
\node at (0,1)[above right] {$i$};
\flipcrossings{2};
\end{knot}

\end{tikzpicture}
\caption{}\label{fig:mapofoutside}
 \end{figure}
 
 When passing from $\Gamma$ to $L(\Gamma,col)$ we still speak of edges and vertices of $L(\Gamma,col)$: we mean the circles corresponding to edges and vertices of $\Gamma$ respectively. Slightly more improperly we speak of faces of $L(\Gamma,col)$, by which we mean the portions of the plane delimited by edges of $\Gamma$. To do this, until the start of Step 3 we fix the diagram of $L(\Gamma,col)$ that we just created.
 
\emph{Step 2:} for a given coloring $col'$ of $\Gamma$, calculate $\mathcal{F}_r(\Gamma,col')$ as the Kauffman bracket of a link $\hat{L}(\Gamma,col')$.  

The Fourier transform is given by the formula

\begin{displaymath}
\mathcal{F}_r(\Gamma,col')=\sum_{col \textrm{ coloring of }\Gamma}\langle L(\Gamma,col)\rangle H(col,col').
 \end{displaymath} 
 
We wish to express this formula as the bracket of a single colored link; to do so we use the following relationship (which can be proven via a particular case of the vertex sum formula from Proposition \ref{prop:vertexsum} after we introduce extra edges colored with $0$):
 
 \begin{equation}
  \sum_{i\in I_r} H(i,j)\vcenter{\hbox{\includegraphics{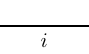}}}=\vcenter{\hbox{\includegraphics{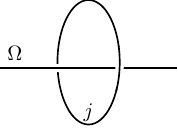}}}.
 \end{equation}

Therefore $\mathcal{F}_r(\Gamma,col')=\langle \hat{L}(\Gamma,col')\rangle$, where $\hat{L}(\Gamma,col')$ is the colored link obtained from $L(\Gamma,col)$ by changing the color of each edge $e$ of $L(\Gamma,col)$ to $\Omega$ and by adding a meridional circle around it colored with $col'(e)$. We call the meridional circles added via this process the \emph{transverse} circles; they will correspond to edges of $\Gamma^*$. Notice that this step only added circles, and did not otherwise change the link diagram we created in Step 1 (not even via planar isotopy).
 
 \begin{figure}
 \begin{minipage}{.49\textwidth}
  \includegraphics[scale=0.12]{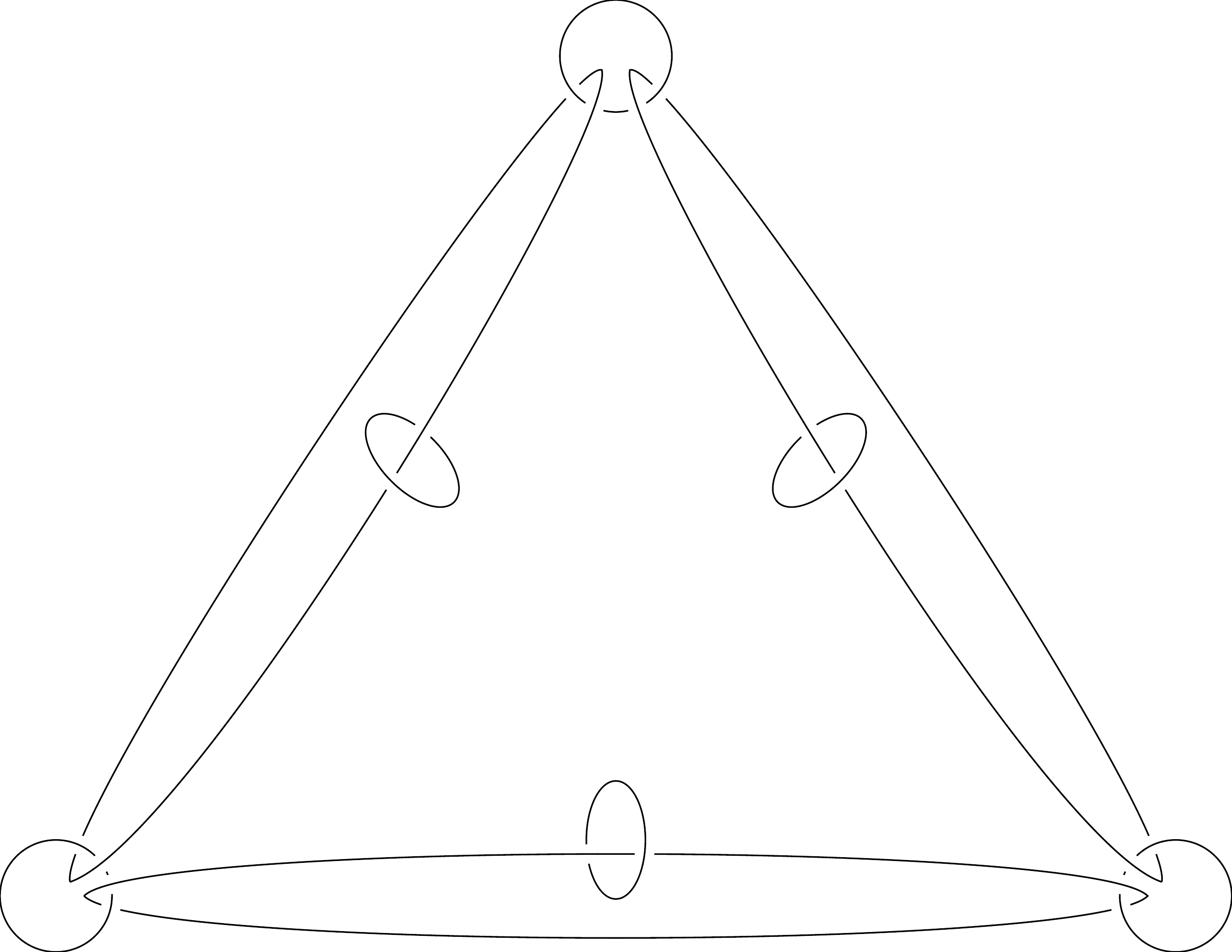}
 \end{minipage}
 \begin{minipage}{.49\textwidth}
  \includegraphics[scale=0.12]{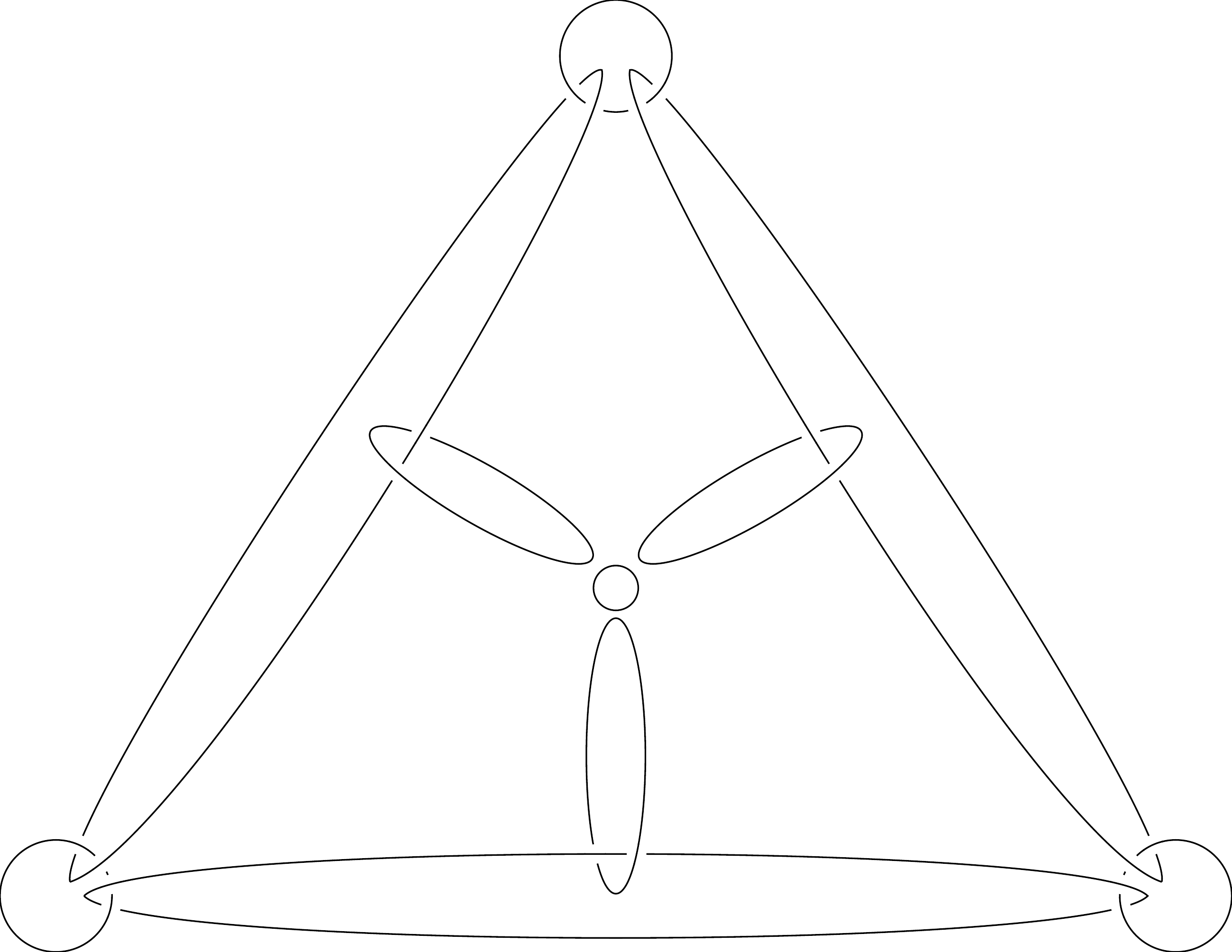}
 \end{minipage}
\caption{Stretching edges towards the center and adding an extra component.}\label{fig:stretching}
\end{figure}

\emph{Step 3:} Manipulate $\hat{L}(\Gamma,col')$ to obtain $L(\Gamma^*,col')$.

 Take a face $F$ of $\hat{L}(\Gamma,col')$ and stretch the circles transverse to its edges so that they are close to the center of $F$ and add an unknot $U$ colored with $\Omega$ at the center of $F$  (see Figure \ref{fig:stretching}). Handleslide this new unknotted component along all the edges of $F$ (see Figure \ref{fig:hs}); the result is that $U$ gets linked to each transversal circle and remains unlinked from any edge or vertex of $\Gamma$ as in Figure \ref{fig:unlinking}. Because the edges are colored with $\Omega$ this procedure does not change the Kauffman bracket (see for example \cite[Corollary, page 181]{lickorish}). The circle $U$ will correspond to a vertex in $\Gamma^*$. Repeat this procedure for every face of $\hat{L}(\Gamma,col')$ (notice also that the procedure we just carried out only changes the link diagram in the portion of the plane corresponding to $F$)

\begin{figure}
\centering
  \begin{tikzpicture}
 \draw [dashed, thick] (0,0) -- (1,0);
 \draw [dashed, thick] (0,2) -- (1,2);
 \draw [thick] (1,0) arc (-90:90:1);
 
 \draw [thick] (4,0) arc (270:90:1);
 \draw [dashed, thick] (4,0) -- (5,0);
 \draw [dashed,thick] (4,2) -- (5,2);
 \draw [->] (2.5,0) -- (2.5,-1);
  \draw [dashed, thick] (0,-3) -- (1,-3);
 \draw [dashed, thick] (0,-1) -- (1,-1);
 \draw [thick] (1,-3) arc (-90:90:1);
 \draw [dashed, thick] (4,-3) -- (5,-3);
 \draw [dashed,thick] (4,-1) -- (5,-1);
 \draw [thick] (4,-1) to [out=180,in=90](3.2,-1.8);
 \draw [thick] (4,-3) to [out=180,in=-90](3.2,-2.2);
 \draw [thick] (3.2,-1.8) -- (2.15,-1.8); 
 \draw [thick] (3.2,-2.2) -- (2.15,-2.2);
 
  \draw [dashed, thick] (0,-3.2) -- (1,-3.2);
 \draw [dashed, thick] (0,-0.8) -- (1,-0.8);
 \draw [thick] (1,-3.2) to [out=0,in=-90](2.15,-2.2);
 \draw [thick] (1,-0.8) to [out=0, in=90] (2.15,-1.8);
 \end{tikzpicture}
 \caption{Handleslide between two different components of $L(\Gamma,col)$}\label{fig:hs}
\end{figure}
  
 Now apply relation \ref{eqn:vertsum} to each circle corresponding to a vertex of $\Gamma$ and each circle corresponding to a vertex of $\Gamma^*$. The result (see Figure \ref{fig:linking}) is going to be four connected graphs (and several unlinked unknots that for now we ignore), which lie in parallel planes and are therefore unlinked from each other. Two of these give $Y_r(\Gamma^*,col')$ and two of these give $Y_r(\Gamma,\Omega)$ (where we still denote with $\Omega$ the coloring of $\Gamma$ with color $\Omega$ on each edge).
  
  \begin{figure}\centering
  \includegraphics[scale=0.12]{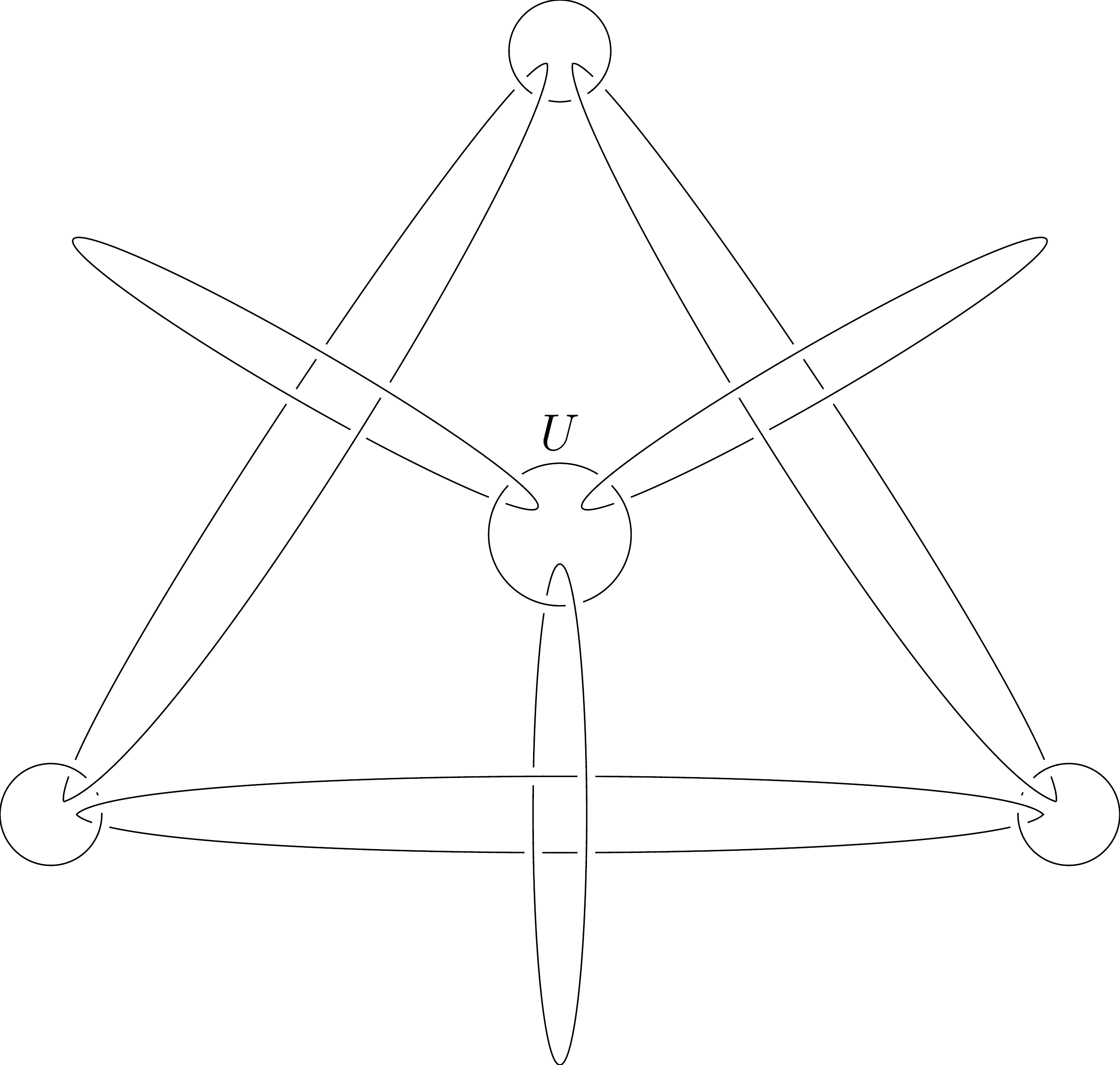}
  \caption{The central component $U$ gets linked by handleslides.}\label{fig:unlinking}
 \end{figure}

  \begin{figure}\centering
  \includegraphics[scale=0.12]{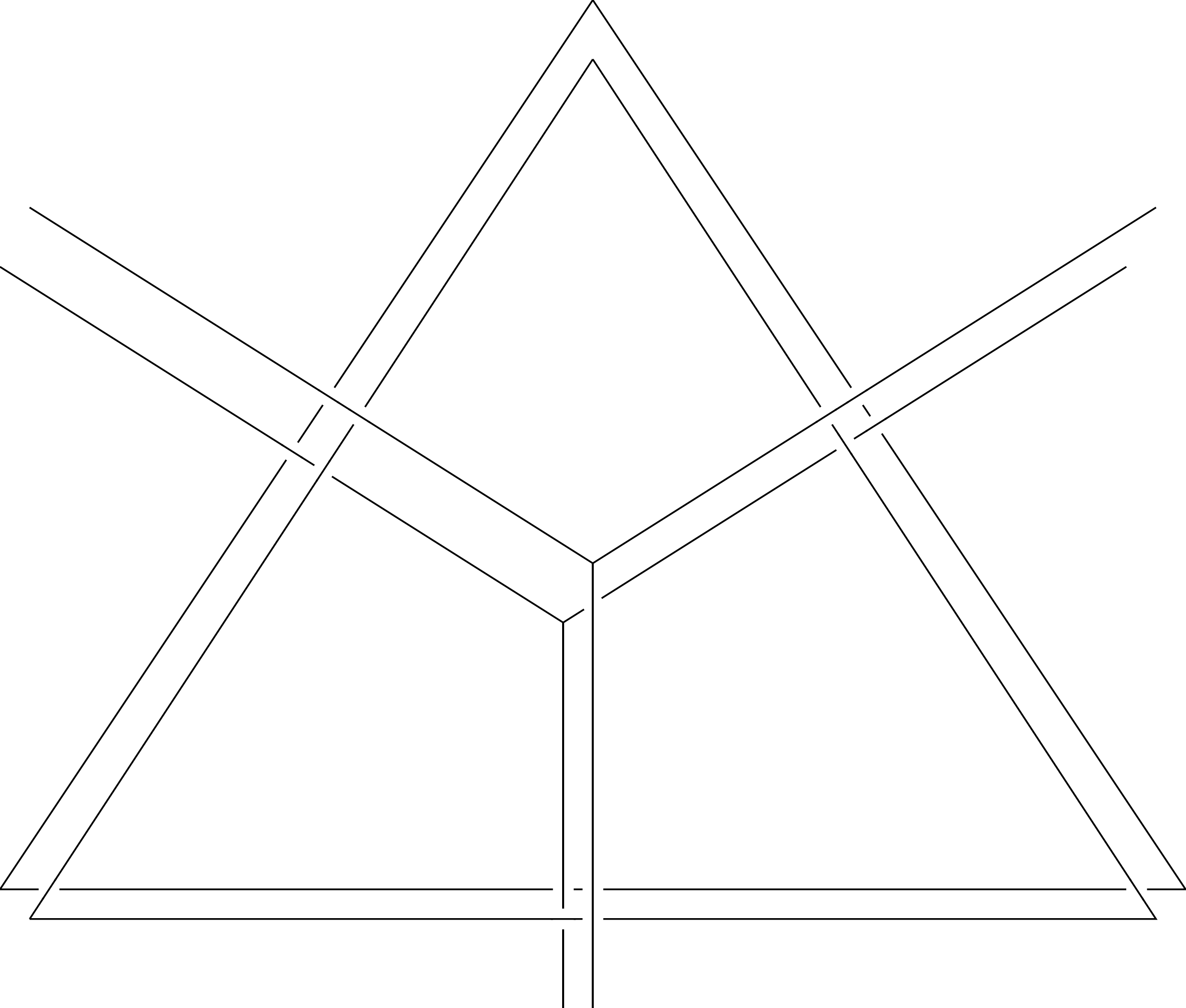}
  \caption{After applying Equation \ref{eqn:vertsum}, we get four unlinked graphs.}\label{fig:linking}
 \end{figure}
\emph{Step 4:} prove that
  $$Y_r(\Gamma,\Omega)=N^g.$$

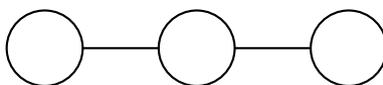
\begin{figure}
\centering
   \begin{tikzpicture}
\centering
\draw[thick](0,0) circle[radius=0.5cm];
\draw[thick](0.5,0)--(1.5,0);
\draw[thick](2,0) circle[radius=0.5cm];
\draw[thick](2.5,0)--(3.5,0);
\draw[thick] (4,0) circle[radius=0.5cm];
\end{tikzpicture}
\caption{Bicycle with $3$ wheels}\label{fig:bycicle}
\end{figure}

To do this, recall that the Yokota invariant does not change when performing a Whitehead move on an edge colored with $\Omega$ (Proposition \ref{prop:yokota}.\ref{prop:whit}). Further recall that a sequence of Whitehead moves can change any trivalent graph into any other trivalent graph with the same number of vertices; this is because:

\begin{itemize}
 \item two trivalent graphs with the same number of vertices also have the same number of faces;
 \item their duals are triangulations with the same number of vertices;
 \item their duals can be changed into one another via ``edge flips'' (see \cite{hatcher});
 \item edge flips are dual to Whitehead moves;
 \item two planar graphs with isotopic duals are themselves isotopic by \cite[Theorem 11]{whitney}.
\end{itemize}

Therefore, we can desingularize and then apply Whitehead moves to $\Gamma$ until it becomes a ``bicycle'' graph as in Figure \ref{fig:bycicle}, with some circles connected linearly by segments; since desingularizing and performing a Whitehead move do not change the genus of the regular neighborhood, there are $g$ circles. Because of the bridge rule \ref{prop:bridge}, the Kauffman bracket is $0$ unless the colors of every connecting edge is $0$, and therefore

\begin{displaymath}
 Y_r(\Gamma,\Omega)=\left(\sum_{i_1,\dots,i_g\in I_r}\Delta_{i_1}^2\cdots\Delta_{i_g}^2\right)=N^g
\end{displaymath}

 \emph{Step 5:} account for the extra $N$ factors.

 At the beginning we added an unknot for each vertex of $\Gamma$, and then for each face. However when we applied the inverse of the chainmail relation we removed the exact same number of components; therefore there is no additional $N$ factor.
 
\end{proof}

\begin{prop}\label{prop:dual}
For any coloring $col$ of a planar graph $\Gamma$,
\begin{displaymath}
 \frac{\pi}{r}\log \abs{Y_r(\Gamma,col)}\leq  \max_{col'}\frac{\pi}{r}\log \abs{Y_r(\Gamma^*,col')} + O\left(\frac{\log r}{r}\right)
\end{displaymath}
where the maximum is taken over all $r$-admissible colorings of the dual graph $\Gamma^*$.
\end{prop}
\begin{proof}

Let $col_{max}$ be an $r$-admissible coloring of $\Gamma^*$ such that $\lvert Y_r(\Gamma^*,col_{max})\rvert$ is maximum.

By Proposition \ref{prop:fourier}, 
\begin{align*}
 &\frac{\pi}{r}\log\abs{Y_r(\Gamma,col)}=\frac{\pi}{r}\log\abs{\sum_{col'}H(col,col'){Y_r(\Gamma^*,col')}}\leq \\ &\frac{\pi}{r}\log\sum_{col'}\abs{H(col,col'){Y_r(\Gamma^*,col_{max})}}=\frac{\pi}{r}\log\abs{Y_r(\Gamma^*,col_{max})}+O\left(\frac{\log r}{r}\right)
\end{align*}

Where the last equality stems from the fact that $\sum_{col'}H(col,col')$ grows polynomially in $r$.
\end{proof}

\begin{cor}
 Conjecture \ref{cng:boundgen} is true for $\Gamma$ if and only if it is true for $\Gamma^*$.

\end{cor}
\begin{proof}
Corollary 4.6 of \cite{maxvol} states that $\vol\left(\overline{\Gamma}\right)=\vol\left(\overline{\Gamma^*}\right)$; this and Proposition \ref{prop:dual} imply the thesis.

\end{proof}

We now turn to the proof of Theorem \ref{teovol}. This will use a few intermediate propositions which we now state and prove.

We first calculate the volume of the rectification of the graphs at hand.

\begin{prop}\label{prop:volrect}
If $\Gamma$ is obtained from the tetrahedron by a sequence of $g$ blow-ups of vertices or triangulations of triangular faces, then
 $$\vol(\overline{\Gamma})=(g+1)v_8.$$
\end{prop}
\begin{proof}
 The case of $g=0$ is well known and appears in \cite[Theorem 4.2]{ushi}.
Take now any $\Gamma$ obtained from $\Gamma'$ by a blow-up of a vertex $v$ and consider $P$ the truncated rectification of $\Gamma'$. The vertex $v$ corresponds to a truncation face of $P$: this face is an ideal triangle. Given an octahedron, we can glue it to $P$ by identifying any of its faces to the face corresponding to $v$ (since they are triangular faces the result does not depend on any choice). Notice that the gluing is done along an ideal triangular face, and along right dihedral angles. It is immediate to see that this gluing gives the truncation of $\overline{\Gamma}$: it has the correct $1$-skeleton (see Figure \ref{fig:blowuprect}) and it is right-angled. Therefore, by blowing up a vertex the maximum volume grows by $v_8$. Dually, triangulating a triangular face makes the maximum volume grow by $v_8$ as well.
 \begin{figure}
 \centering
  \begin{minipage}{.4\textwidth}
  \begin{tikzpicture}
\centering
\draw[thick] (0,1.8)--(0,0)--(-1.8,-1.4);
\draw[thick] (0,0)--(1.8,-1.4);
\draw[color=blue] (0,0.9)--(0.9,-0.7)--(-0.9,-0.7)--(0,0.9);
\end{tikzpicture}
  \end{minipage}
$\xrightarrow{\hspace*{0.8cm}}$
  \begin{minipage}{.4\textwidth}
  \begin{tikzpicture}
\centering
\draw[thick] (0,1.8)--(0,0.9);
\draw[thick] (0.9,-0.7)--(1.8,-1.4);
\draw[thick] (-0.9,-0.7)--(-1.8,-1.4);
\draw[thick] (0,0.9)--(0.9,-0.7)--(-0.9,-0.7)--(0,0.9);
\draw[color=blue] (0,1.8)--(0.45,0.1)--(1.8,-1.4)--(0,-0.7)--(-1.8,-1.4)--(-.45,0.1)--(0,1.8);
\draw[color=blue] (0.45,0.1)--(0,-0.7)--(-.45,0.1)--(.45,.1);
\draw[color=blue] (0,1.8)--(1.8,-1.4)--(-1.8,-1.4)--(0,1.8);
\end{tikzpicture}
  \end{minipage}
  \caption{The $1$-skeleton of the rectification is outlined in blue; a blow-up of a vertex corresponds to gluing an octahedron to its truncation face.}\label{fig:blowuprect}
 \end{figure}
\end{proof}

Next, we prove the upper bound.
\begin{prop}\label{prop:upbound}
If $\Gamma$ is obtained from the tetrahedron by a sequence of $g$ blow-ups of vertices or triangulations of triangular faces, and $col$ is any $r$-admissible coloring, then
\begin{displaymath}
   \frac{\pi}{r}\log\left\lvert Y_r(\Gamma,col)\right\rvert\leq(g+1)v_8+O\left(\frac{\log(r)}{r}\right).
\end{displaymath}
\end{prop}
\begin{proof}
 
The base case $g=0$ is Theorem \ref{teo:bound}.

 If $\Gamma$ is obtained from $\Gamma'$ as a blow-up of a single vertex, then
 \begin{displaymath}
  Y_r(\Gamma,col)=Y_r(\Gamma',col_1)Y_r(T,col_2)
 \end{displaymath}
where $T$ is a tetrahedron, and $col_1,col_2$ are the colorings induced by $col$ on $\Gamma'$ and $T$ respectively. Therefore, $Y_r(\Gamma,col)\leq Y_r(\Gamma',col_1)Y_r(T,col_2)$ and by induction

\begin{displaymath}
\frac{\pi}{r}\log\left\lvert Y_r(\Gamma,col)\right\rvert\leq (g+1)v_8+O\left(\frac{\log(r)}{r}\right).
\end{displaymath}

By Proposition \ref{prop:dual}, this inequality also holds if $\Gamma$ is obtained from $\Gamma'$ by triangulating a single triangular face.

\end{proof}

\begin{figure}
 
\end{figure}

Finally we prove the sharpness of the upper bound.

\begin{prop}\label{prop:r/2}
 If $\Gamma$ is obtained from the tetrahedron by a sequence of $g$ blow-ups of vertices or triangulations of triangular faces, and $col=(\frac{r-2\pm1}{2},\dots,\frac{r-2\pm1}{2})$ (where the signs are chosen so that $r-2\pm1$ is a multiple of $4$), then
 \begin{displaymath}
  \lim_{r\ra+\infty}\frac{\pi}{r}\log\left(Y_r(\Gamma,col)\right)= (g+1)v_8.
 \end{displaymath}

\end{prop}
\begin{proof}

The proof is by induction; the base case is Theorem \ref{teo:bound}. Suppose $\Gamma$ is obtained from the tetrahedron by $g$ blow-ups and triangulations, and at least $1$ blow-up. Then, $\Gamma$ is a vertex sum of $\Gamma_1$ and $\Gamma_2$, with both graphs obtained from the tetrahedron via $g_1$ and $g_2$ blow-ups or triangulations respectively, and $g_1+g_2=g-1$. Since $Y_r(\Gamma,col)=Y_r(\Gamma_1,col_1)Y_r(\Gamma_2,col_2)$ (with $col_1,col_2$ the colorings induced by $col$ on $\Gamma_1,\Gamma_2$ respectively), we have
\begin{align*}
  \lim_{r\ra+\infty}\frac{\pi}{r}\log\left(Y_r(\Gamma,col)\right)&=\lim_{r\ra+\infty}\frac{\pi}{r}\log\left(Y_r(\Gamma_1,col_1)Y_r(\Gamma_2,col_2)\right)=\\ &=(g_1+1+g_2+1)v_8=(g+1)v_8.
 \end{align*}
 
 We need to deal with the case of $\Gamma$ being obtained via $g$ triangulations. In this case, $\Gamma^*$ is obtained from the tetrahedron via $g$ blow-ups.
 Apply the Fourier transform to $Y_r(\Gamma^*,col')$:
 
 \begin{displaymath}
  Y_r(\Gamma,col)=\sum_{col'}H(col,col')Y_r(\Gamma^*,col');
 \end{displaymath}
however, since $col$ is constantly $\frac{r-2\pm1}{2}$ and even, we have

\begin{align*}
 H\left(\frac{r-2\pm1}{2},j\right)=(-1)^j\frac{\sin\left(\frac{2\pi}{r}\frac{r\pm1}{2}(j+1)\right)}{\sin(2\pi/r)}=\\(-1)^j\frac{\sin\left(\pi(j+1)\pm\frac{\pi}{r}(j+1)\right)}{\sin(2\pi/r)}=-\frac{\sin(\pm\frac{\pi}{r}(j+1))}{\sin(2\pi/r)},
\end{align*}
which has $\mp$ sign since $0\leq j\leq r-1$. Moreover, since $\Gamma^*$ is a trivalent graph, $Y_r(\Gamma^*,col')=\lvert\langle \Gamma^*,col'\rangle\rvert^2$ is non-negative for every coloring; therefore, $Y_r(\Gamma,col)$ is a sum with constant sign of $Y_r(\Gamma^*,col')$ over all possible colorings. This shows that $Y_r(\Gamma,col)$ grows as the maximum growth of $Y_r(\Gamma^*,col')$ over all colorings, which is $(g+1)v_8$.
\end{proof}

 Putting Propositions \ref{prop:volrect}, \ref{prop:upbound} and \ref{prop:r/2} together, we obtain the following theorem.
 
\begin{teo}\label{teovol}
 If $\Gamma$ is obtained from the tetrahedron by a sequence of blow-ups of vertices or triangulations of triangular faces, then Conjecture \ref{cng:boundgen} is verified.
\end{teo}

\begin{oss}\label{rmk:volconj}
 Proposition \ref{prop:r/2} actually proves Conjecture \ref{conjmk} for a large family of polyhedra (albeit for a single sequence of colors each) since the volume of a polyhedron with internal angles $0$ is the volume of its rectification (notice how $\frac{2\pi}{r}\frac{r\pm1-2}{2}\ra\pi$ as $r\ra +\infty$). 
\end{oss}

\section{The Turaev-Viro Volume Conjecture}\label{sec:tvvolconj}

In this section we apply Theorem \ref{teovol} to prove the Turaev-Viro Volume Conjecture for an infinite family of examples.

Recall (for example from \cite[Section 4.2]{lickorish}) that the Reshetikhin-Turaev invariant of a colored, framed link $L$ in a manifold $M$ is defined as 
$$RT_r(M,L,col)=\eta \kappa^{\sigma(L')}\langle L\sqcup L', col\cup \Omega\rangle$$
where:
\begin{itemize}
 \item $L'\subseteq S^3$ is a framed link giving $M$ via Dehn surgery,
 \item $L\sqcup L'$ is the disjoint union of $L'$ and $L$ viewed as a subset of $S^3$ (if need be, after isotoping $L$ to be disjoint from $L'$),
 \item the components of $L'$ are all colored with $\Omega$,
 \item $\sigma(L')$ is the signature of the linking matrix of $L'\subseteq S^3$,
 \item $\eta=\langle U,\Omega\rangle^{-1}=\frac{A^2-A^{-2}}{\sqrt{-2r}}$, and
 \item $\kappa=\langle U_+,\Omega\rangle$ where $U_+$ is the unknot with framing equal to $+1$.
\end{itemize}

\begin{prop}\label{prop:sameinv}
 Let $\Gamma\subseteq S^3$ be a graph obtained from the tetrahedron by a sequence of $g-1$ blow-ups of vertices or triangulations of triangular faces as in the hypothesis of Theorem \ref{teovol}; let $e_1,\dots,e_k$ be its edges, and denote with $h$ the number of vertices of $\Gamma$. Then there is a $k$-component link $L=L_1\sqcup\cdots\sqcup L_k$ in $S^3\#^{h-1}\left(S^1\times S^2\right)$ such that for any coloring $col\in I_r^k$ (seen both as a coloring of $\Gamma$ and as a coloring of $L$) we have
 
 $$Y_r(\Gamma,col)=RT_r\left(S^3\#^{h-1}\left(S^1\times S^2\right),L,col\right)$$
 
\end{prop}
\begin{proof}
 We have seen in the proof of Proposition \ref{prop:fourier} (specifically, in Step 1) that there is a way to associate to any $\left(\Gamma,col\right)$ a colored framed link $L(\Gamma,col)$ in $S^3$ such that $Y_r(\Gamma,col)=\langle L(\Gamma,col)\rangle$. The link $L(\Gamma,col)$ is a link with $k+h$ components; $k$ of these components are in bijection with the edges of $\Gamma$ and are colored with the corresponding color of $col$. The other $h$ are unknotted components in bijection with the vertices of $\Gamma$ and are colored with $\Omega$. The link $L(\Gamma,col)$ (without its coloring) almost satisfies the requirements we desire; however it has one too many components.
 
 We now want to remove a component from $L(\Gamma,col)$; we do this so that the end result of the proposition is a  link in $S^3\#^{h-1} S^1\times S^2$ rather than a link in $S^3\#^h S^1\times S^2$.
 
Pick an equatorial $S^2$ in $S^3$ and isotope $L(\Gamma,col)$ so that all its $\Omega$-colored components lie flat on it, and all other components intersect the $S^2$ twice; the fact that this can be done is evident from the construction of $L(\Gamma,col)$. Each $\Omega$-colored component will bound a disk that contains the intersection of its edges with $S^2$; every intersection lands inside one of these disks. Pick a component of $L(\Gamma,col)$ colored with $\Omega$: it is possible to handleslide it along each other $\Omega$-colored component without modifying the Kauffman bracket (by \cite[Corollary, page 181]{lickorish}). Once one such handleslide is performed, the new curve will bound a disk that contains the intersection points of both families of curves. Repeating this procedure and handlesliding the chosen component over all others will make it bound a disk containing all transverse intersection points of $L(\Gamma,col)$ with the plane, thus making it unlinked from everything; therefore $\langle L(\Gamma,col)\rangle=\langle U\rangle\langle L(\Gamma,col)'\rangle=\eta^{-1}\langle L(\Gamma,col)'\rangle$ where $U$ is an unknotted, unlinked component colored with $\Omega$ and $L(\Gamma,col)'$ is the remaining part of the colored link.
 By the definition of the Reshetikhin-Turaev invariant of links $$\langle L(\Gamma,col)'\rangle=\eta RT_r(S^3\#^{h-1}\left(S^1\times S^2\right),L,col)$$ where $L$ is the link obtained from $L(\Gamma,col)'$ by doing a $0$-framed Dehn surgery on the components of $L(\Gamma,col)$ colored with $\Omega$. Notice that $L$ only depends on $\Gamma$ and not on the coloring.
\end{proof}

\begin{dfn}
 We denote the link constructed in Proposition \ref{prop:sameinv} with $K(\Gamma)$ (notice that this is a link rather than a colored link). The next several Propositions explore the geometric properties of $K(\Gamma)$, culminating in proving the Turaev-Viro Volume Conjecture for it.
\end{dfn}

\begin{figure}
 \centering
  \begin{tikzpicture}[scale=0.8]
 \draw[thick][rotate=-45] (0,0) ellipse (1cm and 0.5cm);
 
 \draw[thick][rotate around={-45:(4,4)}] (4,4) ellipse (1cm and 0.5cm);

 \draw[thick][rotate around={45:(4,0)}] (4,0) ellipse (1cm and 0.5cm);
 \draw[thick][rotate around={45:(0,4)}] (0,4) ellipse (1cm and 0.5cm);
\draw [thick](4.71,0.71)[blue] to[out=135, in=-135] (4.71,3.29);

\draw [thick](0.71,-0.71)[red] to[out=45, in=135] (3.29,-0.71);
\draw [thick](-0.71,0.71)[green] to[out=45, in=-45] (-0.71,3.29);
\draw [thick](0.71,4.71)[orange] to[out=-45, in=-135] (3.29,4.71);
\draw [thick](0.36,0.36)[brown] -- (3.64,3.64);
\draw [thick](-0.36,4.36)[magenta,dashed]--(4.36,-0.36);
\end{tikzpicture}
\caption{The building block: a ball with $4$ disks in its boundary, and $6$ arcs connecting them.}\label{fig:bblock}
\end{figure}
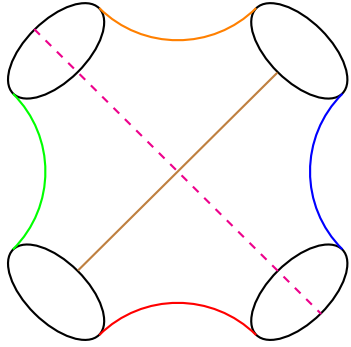

\begin{prop}\label{prop:volnuovo}
 Let $\Gamma\subseteq S^3$ be a graph obtained from the tetrahedron by a sequence of $g-1$ blow-ups of vertices or triangulations of triangular faces; suppose $\Gamma$ has $k$ edges and $h$ vertices. Then, $L:=K\left(\Gamma\right)$ is hyperbolic, and the hyperbolic structure on its complement is obtained by gluing $4g$ right-angled hyperbolic ideal octahedra.
\end{prop}
\begin{proof}
 Let $\overline{\Gamma}$ be the rectification of $\Gamma$, and let $P$ be its truncation. We have seen in the proof of Theorem \ref{teovol} that $P$ can be obtained by gluing $g$ right-angled hyperbolic octahedra. Take two copies of $P$ and glue them along each corresponding truncation face. This gives a manifold homeomorphic to a handlebody of genus $h-1$ with some annuli removed from the boundary (corresponding to the ideal vertices of $P$); the decomposition into octahedra makes it into a finite volume manifold $M$ with geodesic boundary. Take the double of $M$ along the geodesic boundary: this gives a manifold $N$ which is homeomorphic to $S^3\#^{h-1}\left(S^1\times S^2\right)\setminus L$. 
 
 To see this, take an octahedron $O$ and truncate a small link of each of its vertices. This truncation can be seen as the basic building block of the fundamental shadow links (see Figure \ref{fig:bblock} and \cite[Proposition 3.33]{CosThurston}): each truncated vertex corresponds to an arc, four of the faces of the octahedron correspond to the discs and the remaining four faces correspond to the regions of the spheres delimited by the arcs. 
 
 The polyhedron $P$ is obtained by gluing octahedra together following a certain pattern; we can glue the building blocks in the same pattern. The result of this gluing is a ball with $h$ discs on its boundary and some arcs connecting the discs. If we take the double of this ball along the discs we obtain a genus $h-1$ handlebody with a link in its boundary. The link $L(\Gamma,col)$ corresponds to the link on the boundary of the handlebody plus $h-1$ components corresponding to the boundary of the gluing disks (after pushing them out of the handlebody slightly).
 
 Doubling this handlebody is equivalent to performing $0$-surgery on each of these $h-1$ components in $S^3$, therefore by doing this we obtain $S^3\#^{h-1} (S^1\times S^2)$ as ambient manifold and the link in the boundary gives $L$.
 
\end{proof}

\begin{prop}\label{prop:diff}
 Let $\Gamma$ be a graph obtained from the tetrahedron by a sequence of $g-1$ blow-ups of vertices or triangulations of triangular faces; let $t$ be the maximal number of disjoint triangular faces in the truncation of $\overline{\Gamma}$. Let $L:=K(\Gamma)$, and $E_L$ be its complement. Then $E_L$ contains at most $t+2g-2$ disjoint geodesic thrice-punctured spheres.
\end{prop}
\begin{proof}
The reasoning in this proof is similar to the proof of \cite[Proposition 3.4]{CPFM}.

Let $P$ be the truncation of $\overline{\Gamma}$; we have seen in the proof of Proposition \ref{prop:volnuovo} that $E_L$ is obtained by doubling $P$ along the truncation faces (to obtain a hyperbolic manifold with geodesic boundary $H$) and doubling again along the geodesic boundary.

 The truncation faces of $P$ can be colored with black and the remaining with white; this way two faces of the same color never share an edge.
 
 The proof of Proposition \ref{prop:volnuovo} shows that $E_L$ decomposes into octahedra; take $O$ an octahedron in this decomposition, and let $S$ be any thrice-punctured sphere.
 
 \begin{figure}
  \centering
  
  \begin{tikzpicture}[scale=0.7]
\centering
\draw[thick] (0,0)--(4,0)--(2,3.4)--(0,0);
\draw[thick] (-.8,-.45)--(4,2.2);
\draw[thick] (4.8,-0.45)--(0,2.2);
\draw[thick] (2,3.4)--(2,-.5);
\draw[color=blue, fill=blue] (0,0) circle[radius=0.1cm];
\draw[color=blue, fill=blue] (4,0) circle[radius=0.1cm];
\draw[color=blue, fill=blue] (2,3.4) circle[radius=0.1cm];
\end{tikzpicture}
\caption{The $6$ geodesic in a thrice punctured sphere cutting it into triangles.}\label{fig:cutup}
 \end{figure}
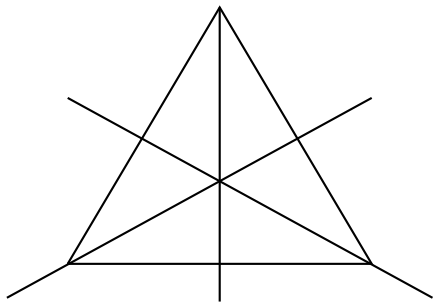

 \textit{Claim:} $S\cap O$ is either the empty set or a facet of $O$.
 
 We will prove the claim later; for now let us see how this concludes the proof.
 
 Let $\mathcal{S}$ be a set of disjoint thrice-punctured spheres. This determines a set of disjoint ideal triangles in each of the four copies of $P$ that make up $E_L$; some of them are in the boundary of a polyhedron while some of them are properly embedded. Each polyhedron contains exactly $g-1$ disjoint properly embedded geodesic triangles (the ones that decompose $P$ into octahedra). These glue up to give $2g-2$ disjoint thrice-punctured spheres in $E_L$. Furthermore, a disjoint collection $T_1,\dots,T_t$ of triangles in $\partial P$ induces a set of disjoint thrice-punctured spheres. Therefore, there are at most $2g-2+t$ disjoint thrice-punctured spheres in $E_L$.

 \textit{Proof of the claim.}

 We first look at $S\cap O$ as a subset of $S$. It must be a convex region of $S$ delimited by geodesics. Since $S$ contains exactly $6$ maximal embedded geodesics (since it contains no closed geodesics and maximal embedded geodesics are determined by the cusp in which they end) the possible configurations are easy to list. Figure \ref{fig:cutup} shows the $6$ geodesics cutting $S$ into triangles; the possibilities for $S\cap O$ can be obtained by looking at all the possible ways to glue these triangles to obtain a convex set. The convex subsets of $S$ obtained by gluing triangle regions are:
 
 \begin{enumerate}
  \item \label{ref:case1} a triangle with $1$ ideal vertex (a single triangle region);
  \item \label{ref:case2} a triangle with $2$ ideal vertices (gluing two triangle regions without an ideal vertex in common);
  \item \label{ref:case3} a square with one ideal vertex and two right angles (gluing two triangle regions with an ideal vertex in common);
  \item \label{ref:case4} a triangle with two ideal vertices and a right angle (gluing a triangle region to the triangle in Case \ref{ref:case2});
  \item \label{ref:case5} a square with two ideal vertices (gluing two triangles in Case \ref{ref:case2} along a geodesic);
  \item \label{ref:case6} a bigon with one ideal point in its interior (gluing all triangle regions sharing an ideal vertex);
  \item \label{ref:case7} a bigon with one ideal point in its boundary (gluing two triangle regions that have all the edges on the same geodesics);
  \item \label{ref:case8}a region with three ideal points (obtained in several possible ways).
 \end{enumerate}

 Every other possible way of gluing together the triangle regions of Figure \ref{fig:cutup} does not give a convex subset.

 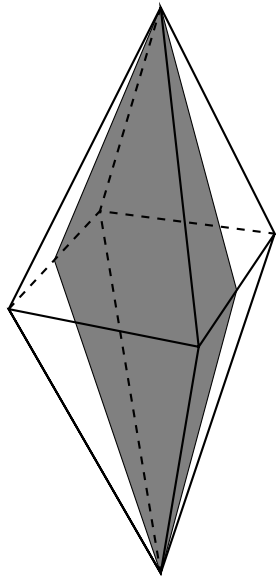
\begin{figure}
  \centering
  \begin{tikzpicture}[scale=0.7]
  \centering
  
  \draw[fill=gray] (2,4)--(3,0.25)--(2,-3.5)--(0.6,0.65)--(2,4);
  \draw[thick] (2,4)--(0,0)--(2,-3.5)--(2.5,-.5)--(3.5,1)--(2,4)--(2.5,-.5)--(0,0)--(2,-3.5)--(3.5,1);
  \draw[thick, dashed] (0,0)--(1.2,1.3)--(3.5,1);
  \draw[thick,dashed](2,4)--(1.2,1.3)--(2,-3.5);
  
\end{tikzpicture}
\caption{A square arising as the intersection of a thrice-punctured sphere and an octahedron of $E_L$.}\label{fig:intersection}
 \end{figure}
 On the other hand, $S\cap O$ as a subset of $O$ must coincide with the intersection of $O$ with a plane $\Pi\subseteq \h$; therefore it cannot be either a bigon with an ideal point in its interior (Case \ref{ref:case6}) or a bigon with an ideal point in its boundary (Case \ref{ref:case7}), since these regions cannot be realized as a hyperbolic polygon in $\h$ (hence, neither in $O$). Moreover, $\Pi\cap O$ cannot be a triangle with one or two ideal vertices (this excludes Cases \ref{ref:case1}, \ref{ref:case2} and \ref{ref:case4}), nor can it be a square with one ideal vertex and two right angles (Case \ref{ref:case3}), since none of these configurations can be realized as intersections of a plane with $O$. The remaining possibilities are that $S\cap O$ is a region with $3$ ideal points (Case \ref{ref:case8}), a square with two ideal vertices (Case \ref{ref:case5}: see Figure \ref{fig:intersection}), or a facet of dimension $0$ or $1$. However by construction $O$ is glued to at least three octahedra which are different from $O$ and each other; therefore the case of a square with two ideal vertices is impossible since the intersection of $S$ with these octahedra must also be a square with $2$ ideal vertices, which would contradict the fact that $S$ is a thrice-punctured sphere. Finally there are no properly embedded, totally geodesic surfaces with exactly $3$ ideal points in $O$; therefore the only possibility is that it is a face of $O$. To sum up, the only possible cases are that $S\cap O$ (when non-empty) is a vertex, an edge, or a face of $O$, therefore $S\cap O$ must be a facet of $O$. This concludes the proof of the claim.

\end{proof}

\begin{oss}\label{rmk:newmanifolds}
 If $M$ is the exterior of a fundamental shadow link with volume $2nv_8$, then it contains exactly $2n$ disjoint geodesic thrice-punctured spheres. This can be used to show that some of the exterior of the links provided by Proposition \ref{prop:sameinv} are not homeomorphic to exteriors of fundamental shadow links; the simplest such example is the link associated to the graph shown in Figure \ref{fig:example}. An easy check shows that the truncation of $\overline{\Gamma}$ contains at most $6$ disjoint triangular faces: they correspond to the truncation faces of the three vertices on the left half of the picture, and to the three triangular faces on the right half of the picture. This means that (by Proposition \ref{prop:diff}) $E_L$ contains at most $10$ thrice-punctured spheres and has volume $12 v_8$; on the other hand a fundamental shadow link complement with the same volume as $E_L$ must contain $12$ such spheres.
 
 More generally, if $\Gamma$ is obtained from the tetrahedron through at least one triangulation and at least one blow-up, then the associated link exterior is not diffeomorphic to the exterior of a fundamental shadow link (and there is at least one such manifold of volume $4nv_8$ for each $n>1$).
\end{oss}
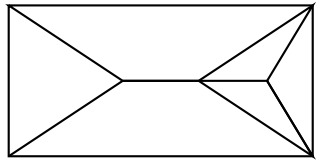
\begin{figure}
 \centering
  \begin{tikzpicture}
  \centering
  \draw[thick] (0,2)--(0,0)--(4,0)--(4,2)--(0,2)--(1.5,1)--(3.4,1)--(4,0);
  \draw[thick] (0,0)--(1.5,1)--(2.5,1)--(4,2)--(3.4,1)--(4,0)--(2.5,1);
  
\end{tikzpicture}
\caption{A graph whose link is not a fundamental shadow link.}\label{fig:example}
\end{figure}

\begin{teo}\label{teo:volconj}
 Let $\Gamma\subseteq S^3$ be a graph obtained from the tetrahedron by a sequence of $g-1$ blow-ups of vertices or triangulations of triangular faces; suppose $\Gamma$ has $k$ edges and $h$ vertices. Take the $k$-component link $L:=K(\Gamma)\subseteq S^3\#^{h-1}\left(S^1\times S^2\right)$. Then the Turaev-Viro Volume Conjecture (Conjecture \ref{volconj}) holds for the exterior of $L$.
\end{teo}
\begin{proof}
 Theorem \ref{teovol} implies that for any choice of $r$-admissible coloring $col$, $$\frac{\pi}{r}\log\lvert RT_r(S^3\#^{h-1}\left(S^1\times S^2\right),L,col)\rvert=\frac{\pi}{r}\log\lvert Y_r(\Gamma,col)\rvert \leq gv_8+O\left(\frac{\log(r)}{r}\right).$$
 
 The equality is a consequence of Proposition \ref{prop:sameinv}; the subsequent inequality is the content of Theorem \ref{teovol}.
 Furthermore if we denote with $c$ the coloring $\left(\frac{r\pm1}{2},\dots,\frac{r\pm1}{2}\right)$ (where the sign is chosen so that the color is always even), we have 
 $$\frac{\pi}{r}\log\lvert RT_r(S^3\#^{h-1}\left(S^1\times S^2\right),L,c)\rvert=\frac{\pi}{r}\log\lvert Y_r(\Gamma,c)\rvert = gv_8+O(\log(r)/r)$$
 because of Proposition \ref{prop:r/2}.
 
 If $E_L$ is the exterior of $L$, then
 $$TV_r(E_L)=\sum_{col\in I_r^k}\lvert RT_r(S^3\#^{h-1}\left(S^1\times S^2\right),L,col)\rvert^2$$ by \cite[Proposition 5.3]{bound6j}, and $\vol(E_L)=4gv_8$ by Proposition \ref{prop:volnuovo}. 
 
 This implies the thesis since
 
 \begin{displaymath}
 \lim_{r\ra\infty} \frac{2\pi}{r}\log\left(TV_r(E_L)\right)=4gv_8
 \end{displaymath}
because the sum in the formula for $TV_r(E_L)$ has polynomially many terms all with the same sign.
\end{proof}

\begin{oss}
 There is an overlap between Theorem \ref{teo:volconj} and Theorem 1.1 of \cite{bound6j}. Some links of Theorem \ref{teo:volconj} are also Fundamental Shadow Links (FSL); namely, those links corresponding to graphs obtained from the tetrahedron by blow-ups. However as we have seen in Remark \ref{rmk:newmanifolds} (infinitely) many others are not.
\end{oss}

\appendix

\section{Appendix: Numerical evidence for Conjecture \ref{conjmk}}\label{appendix}

Supporting evidence for Conjecture \ref{conjmk} in the case of simple polyhedra can be found in \cite{murkolp}. In this appendix we show numerical computations supporting the conjecture for the square and pentagonal pyramids; all the calculations are performed with the Mathematica software. The notebook is available on GitHub at \url{https://github.com/Giulio451/UpperBound}; all calculations were performed on a Dell XPS 13 laptop.

\emph{The ideal regular square pyramid.} 

By Bao-Bonahon (\cite[Theorem 1]{bonbao}) there is a unique square pyramid such that the angles at the base are $\frac{\pi}{4}$ and the vertical angles are $\frac{\pi}{2}$. Such a pyramid is ideal and is maximally symmetric; it is decomposed into two ideal tetrahedra with angles $\frac{\pi}{4},\frac{\pi}{4},\frac{\pi}{2}$ hence its hyperbolic volume is equal to $4\Lambda\left(\frac{\pi}{4}\right)=\frac{v_8}{2}\sim 1.83193$ (where $\Lambda$ is the Lobachevski function). Consider the coloring of Figure \ref{fig:sqpyr}; it converges to the angles of the ideal pyramid in the sense of Conjecture \ref{conjmk}.

\begin{figure}\centering
   \begin{tikzpicture}
\centering
\draw[thick](0,0)--node[above left]{$\floor*{\frac{r}{4}}$}(2,2);
\draw[thick](4,0)--node[above right]{$\floor*{\frac{r}{4}}$}(2,2);
\draw[thick](0,4)--node[below left]{$\floor*{\frac{r}{4}}$}(2,2);
\draw[thick](4,4)--node[below right]{$\floor*{\frac{r}{4}}$}(2,2);
\draw[thick](0,0)--node[below]{$\floor*{\frac{3r}{8}}$}(4,0);
\draw[thick](4,4)--node[right]{$\floor*{\frac{3r}{8}}$}(4,0);
\draw[thick](0,0)--node[left]{$\floor*{\frac{3r}{8}}$}(0,4);
\draw[thick](0,4)--node[above]{$\floor*{\frac{3r}{8}}$}(4,4);
\end{tikzpicture}
\caption{The coloring of a square pyramid associated to the ideal regular pyramid}\label{fig:sqpyr}
\end{figure}
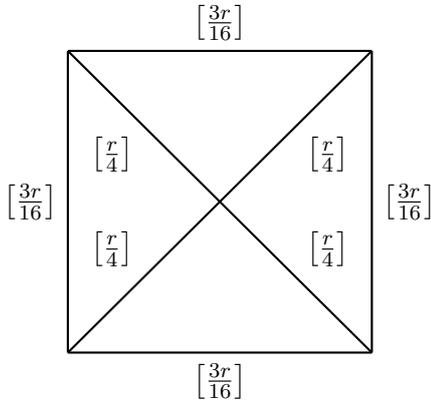

Its Yokota invariant can be calculated by desingularizing the $4$-valent vertex and by using the vertex sum formula; it is given by

\begin{displaymath}
 \sum_{k\in I_r}\Delta_k\begin{vmatrix}
   \floor*{\frac{r}{4}} &\floor*{\frac{r}{4}} &k\\
   \floor*{\frac{3r}{8}} &\floor*{\frac{3r}{8}} &\floor*{\frac{3r}{8}} 
  \end{vmatrix}^4
\end{displaymath}

where $\floor*{x}$ is the floor of $x$. The growth is shown in the following graph.

\includegraphics[width=.87\textwidth]{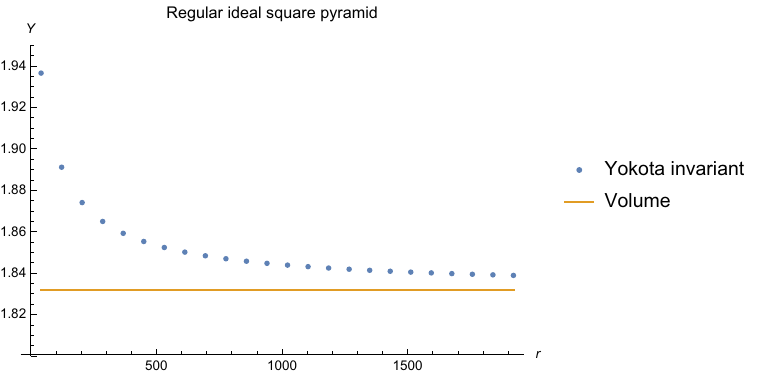}

\emph{The $0$-angled squared pyramid}

Because of \cite[Theorem 4.2]{maxvol}, the square pyramid with every dihedral angle equal to $0$ exists and attains the maximum volume of any square pyramid (it is in fact the rectified pyramid). Its truncation is the right-angled ideal square antiprism. The volume of a right-angled ideal antiprism with $n$-gonal face is given by \cite[Page 151]{thurston} $$2n\left(\Lambda\left(\frac{\pi}{4}+\frac{\pi}{2n}\right)+\Lambda\left(\frac{\pi}{4}-\frac{\pi}{2n}\right)\right)$$

and for $n=4$ this gives $\sim 6.02305$. 

Color the pyramid with the color $\floor*{\frac{r}{2}}$ at every vertex; this coloring converges to the angles of the rectified pyramid.

Its Yokota invariant is given by

\begin{displaymath}
 \sum_{k\in I_r}\Delta_k\begin{vmatrix}
   \floor*{\frac{r}{2}} &\floor*{\frac{r}{2}} &k\\
   \floor*{\frac{r}{2}} &\floor*{\frac{r}{2}} &\floor*{\frac{r}{2}} 
  \end{vmatrix}^4
\end{displaymath}

and its growth is shown in the following graph.

\includegraphics[width=.87\textwidth]{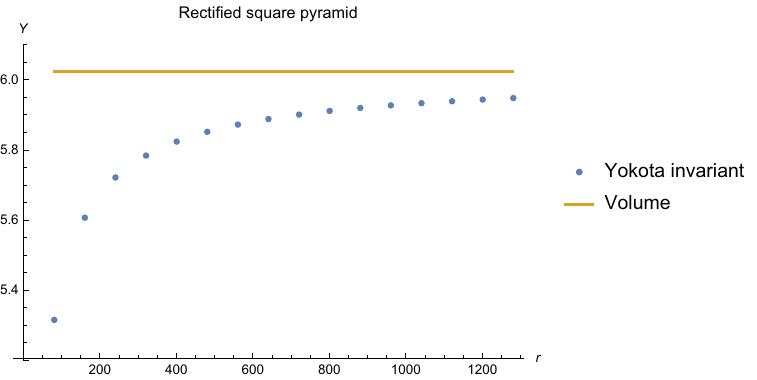}

\emph{The ideal regular pentagonal pyramid.} 

As before there is a unique ideal pentagonal pyramid with vertical angles $\frac{3\pi}{5}$ and base angles $\frac{\pi}{5}$; this pyramid is maximally symmetric. We can decompose it into $3$ ideal tetrahedra, two with dihedral angles $\frac{\pi}{5},\frac{\pi}{5},\frac{3\pi}{5}$ and the remaining with dihedral angles $\frac{\pi}{5},\frac{2\pi}{5},\frac{2\pi}{5}$. Its volume then is $$5\Lambda\left(\frac{\pi}{5}\right)+2\Lambda\left(\frac{2\pi}{5}\right)+\Lambda\left(\frac{3\pi}{5}\right)\sim2.49339.$$

\begin{figure}
\centering
 \begin{tikzpicture}
  \centering
  \draw[thick](0,0)--node[above left]{$\floor*{\frac{r}{5}}$}(2,2);
\draw[thick](4,0)--node[above right]{$\floor*{\frac{r}{5}}$}(2,2);
\draw[thick](-0.5,3.5)--node[below left]{$\floor*{\frac{r}{5}}$}(2,2);
\draw[thick](4.5,3.5)--node[below right]{$\floor*{\frac{r}{5}}$}(2,2);
\draw[thick](2,5)--node[right]{$\floor*{\frac{r}{5}}$}(2,2);
\draw[thick](0,0)--node[below]{$\floor*{\frac{2r}{5}}$}(4,0);
\draw[thick](4.5,3.5)--node[right]{$\floor*{\frac{2r}{5}}$}(4,0);
\draw[thick](0,0)--node[left]{$\floor*{\frac{2r}{5}}$}(-0.5,3.5);
\draw[thick](-0.5,3.5)--node[above]{$\floor*{\frac{2r}{5}}$}(2,5);
\draw[thick](4.5,3.5)--node[above]{$\floor*{\frac{2r}{5}}$}(2,5);
 \end{tikzpicture}
\caption{The coloring of the pentagonal pyramid corresponding to an ideal regular pyramid}\label{fig:pentpyr}
\end{figure}
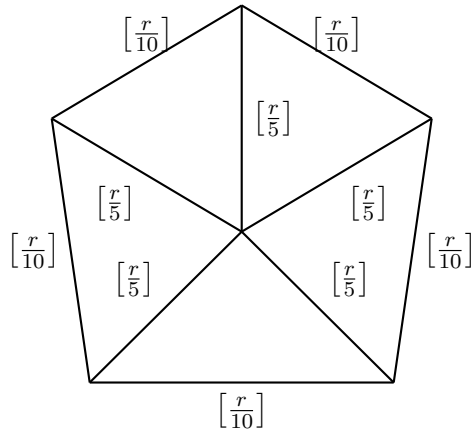

Consider the coloring in Figure \ref{fig:pentpyr}, converging to the angles of the ideal pyramid. Its Yokota invariant can be calculated (by desingularization and the vertex sum formula) as

\begin{displaymath}
 \sum_{k,j\in I_r} \Delta_k\Delta_j\left|\left(\begin{vmatrix}
   \floor*{\frac{2r}{5}} &\floor*{\frac{2r}{5}} &k\\
   \floor*{\frac{r}{5}} &\floor*{\frac{r}{5}} &\floor*{\frac{r}{5}} 
  \end{vmatrix}\begin{vmatrix}
   \floor*{\frac{2r}{5}} &\floor*{\frac{2r}{5}} &j\\
   \floor*{\frac{r}{5}} &\floor*{\frac{r}{5}} &\floor*{\frac{r}{5}} 
  \end{vmatrix}\begin{vmatrix}
   \floor*{\frac{2r}{5}} &k &j\\
   \floor*{\frac{r}{5}} &\floor*{\frac{r}{5}} &\floor*{\frac{r}{5}} 
  \end{vmatrix}\right)\right|^2
\end{displaymath}

and its growth is shown in the following graph.

\includegraphics[width=.88\textwidth]{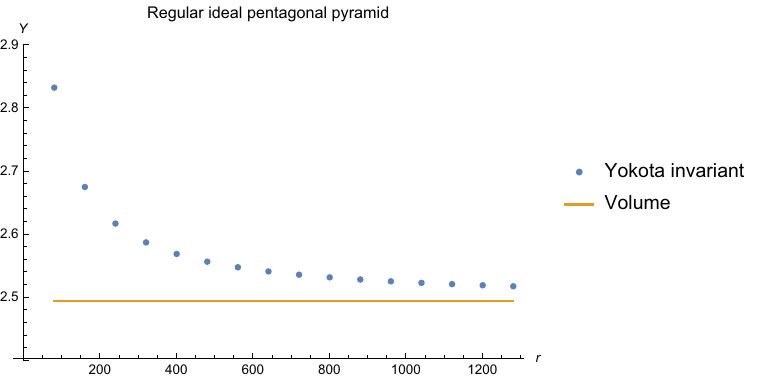}

\emph{The $0$-angled pentagonal pyramid}

The truncation of the rectified pentagonal pyramid is the pentagonal antiprism, whose volume is equal to $\sim8.13789$, and the corresponding Yokota invariant is

\begin{displaymath}
 \sum_{k,j\in I_r} \Delta_k\Delta_j\left|\left(\begin{vmatrix}
   \floor*{\frac{r}{4}} &\floor*{\frac{r}{4}} &k\\
   \floor*{\frac{r}{4}} &\floor*{\frac{r}{4}} &\floor*{\frac{r}{4}} 
  \end{vmatrix}\begin{vmatrix}
   \floor*{\frac{r}{4}} &\floor*{\frac{r}{4}} &j\\
   \floor*{\frac{r}{4}} &\floor*{\frac{r}{4}} &\floor*{\frac{r}{4}} 
  \end{vmatrix}\begin{vmatrix}
   \floor*{\frac{r}{4}} &k &j\\
   \floor*{\frac{r}{4}} &\floor*{\frac{r}{4}} &\floor*{\frac{r}{4}} 
  \end{vmatrix}\right)\right|^2.
\end{displaymath}

Because of the greater range of the sum, it is considerably slower to compute than the other examples; we were only able to arrive to level $r=321$, and the Yokota invariant is within $4\%$ of the volume, as can be seen from the following graph. However this is similar to the error (at level $321$) in the previous examples.

\includegraphics[width=.86\textwidth]{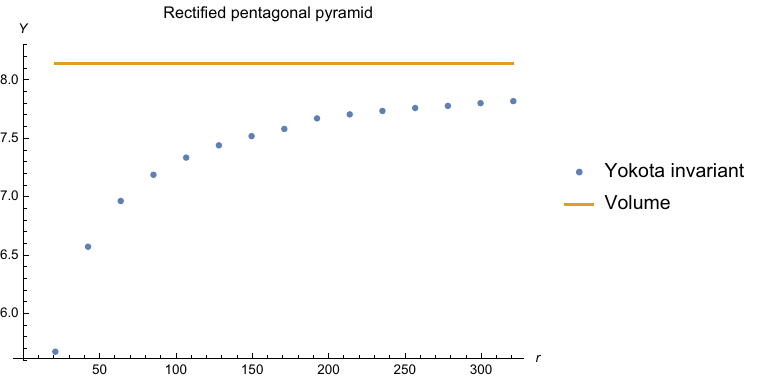}

\bibliographystyle{alpha}
\bibliography{Bibliography}

\def\cprime{$'$}
\begin{thebibliography}{BFMGI07}

\bibitem[Bar03]{bar}
J.W. Barrett.
\newblock Geometrical measurements in three-dimensional quantum gravity.
\newblock {\em Internat. J. Modern Phys. A}, 18(supp02):97--113, 2003.

\bibitem[BB02]{bonbao}
X.~Bao and F.~Bonahon.
\newblock Hyperideal polyhedra in hyperbolic 3-space.
\newblock {\em Bull. Soc. Math. de France}, 130(3):457--491, 2002.

\bibitem[BDKY22]{bound6j}
G.~Belletti, R.~Detcherry, E.~Kalfagianni, and T.~Yang.
\newblock Growth of quantum $6 j $-symbols and applications to the volume
  conjecture.
\newblock {\em Journal of differential geometry}, 120(2):199--229, 2022.

\bibitem[Bel21]{maxvol}
G.~Belletti.
\newblock The maximum volume of hyperbolic polyhedra.
\newblock {\em Transactions of the American Mathematical Society},
  374(2):1125--1153, 2021.

\bibitem[BFMGI07]{bar+}
J.W. Barrett, J.~Faria~Martins, and J.M. Garc{\'\i}a-Islas.
\newblock {Observables in the Turaev-Viro and Crane-Yetter models}.
\newblock {\em J. Math. Phys.}, 48(9):093508, 2007.

\bibitem[CFMP07]{CPFM}
F.~Costantino, R.~Frigerio, B.~Martelli, and C.~Petronio.
\newblock Triangulations of 3-manifolds, hyperbolic relative handlebodies, and
  {D}ehn filling.
\newblock {\em Comment. Math. Helv.}, 82(4):903--933, 2007.

\bibitem[CGvdV15]{volconjpoly}
F.~Costantino, F.~Gu{\'e}ritaud, and R.~van~der Veen.
\newblock On the volume conjecture for polyhedra.
\newblock {\em Geom. Dedicata}, 179(1):385--409, 2015.

\bibitem[CM]{chenmur}
Q.~Chen and J.~Murakami.
\newblock {Asymptotics of Quantum 6$j$-Symbols}.
\newblock {\em arXiv preprint math.GT/1706.04887}.

\bibitem[Cos07]{C}
F.~Costantino.
\newblock {$6j$}-symbols, hyperbolic structures and the volume conjecture.
\newblock {\em Geom. Topol.}, 11:1831--1854, 2007.

\bibitem[CT08]{CosThurston}
F.~Costantino and D.~Thurston.
\newblock 3-manifolds efficiently bound 4-manifolds.
\newblock {\em J. Topol.}, 1(3):703--745, 2008.

\bibitem[CY18]{cyvolconj}
Q.~Chen and T.~Yang.
\newblock {Volume conjectures for the Reshetikhin--Turaev and the Turaev--Viro
  invariants}.
\newblock {\em Quantum Topol.}, 9(3):419--460, 2018.

\bibitem[DKY18]{DKY}
R.~Detcherry, E.~Kalfagianni, and T.~Yang.
\newblock {Turaev-Viro invariants, colored Jones polynomials, and volume}.
\newblock {\em Quantum Topol.}, 9(4):775--813, 2018.

\bibitem[Fle73]{fle}
H.~Fleischner.
\newblock The uniquely embeddable planar graphs.
\newblock {\em Discrete Math.}, 4(4):347--358, 1973.

\bibitem[Hat91]{hatcher}
A.~Hatcher.
\newblock On triangulations of surfaces.
\newblock {\em Topology and its Applications}, 40(2):189--194, 1991.

\bibitem[KL94]{kauflins}
L.H. Kauffman and S.~Lins.
\newblock {\em {Temperley-Lieb recoupling theory and invariants of
  3-manifolds}}.
\newblock Princeton University Press, 1994.

\bibitem[KM18]{murkolp}
A.~Kolpakov and J.~Murakami.
\newblock {Combinatorial Decompositions, Kirillov--Reshetikhin Invariants, and
  the Volume Conjecture for Hyperbolic Polyhedra}.
\newblock {\em Experimental Mathematics}, 27(2):193--207, 2018.

\bibitem[Lic93]{lickorish}
W.B.R. Lickorish.
\newblock The skein method for three-manifold invariants.
\newblock {\em J. Knot Theory Ramifications}, 2(02):171--194, 1993.

\bibitem[Oht18]{ohtdf}
T.~Ohtsuki.
\newblock {On the asymptotic expansion of the quantum $SU(2)$ invariant at
  $q=\exp(4\pi\sqrt{-1}/N)$ for closed hyperbolic $3$-manifolds obtained by
  integral surgery along the figure-eight knot}.
\newblock {\em Algebr. Geom. Topol.}, 18(7):4187--4274, 2018.

\bibitem[RH93]{rivhodg}
I.~Rivin and C.~Hodgson.
\newblock A characterization of compact convex polyhedra in hyperbolic 3-space.
\newblock {\em Invent. Math.}, 111(1):77--111, 1993.

\bibitem[Ste22]{steinitz}
E.~Steinitz.
\newblock Polyeder und raumeinteilungen.
\newblock {\em Encyk. der Math. Wiss.}, 12:38--43, 1922.

\bibitem[Thu79]{thurston}
W.~Thurston.
\newblock {\em The geometry and topology of three-manifolds}.
\newblock Princeton University Princeton, NJ, 1979.

\bibitem[TV92]{TV}
V.~G. Turaev and O.~Y. Viro.
\newblock {State sum invariants of {$3$}-manifolds and quantum {$6j$}-symbols}.
\newblock {\em {Topology}}, 4:865---902, 1992.

\bibitem[Ush06]{ushi}
A.~Ushijima.
\newblock A volume formula for generalised hyperbolic tetrahedra.
\newblock In {\em Non-Euclidean geometries}, pages 249--265. Springer, 2006.

\bibitem[vdV09]{vdv}
R.~van~der Veen.
\newblock The volume conjecture for augmented knotted trivalent graphs.
\newblock {\em Algebr. Geom. Topol.}, 9(2):691--722, 2009.

\bibitem[Whi31]{whitney}
H.~Whitney.
\newblock Non-separable and planar graphs.
\newblock {\em Proceedings of the National Academy of Sciences},
  17(2):125--127, 1931.

\bibitem[WY20]{wongyangvolume}
K.~H. Wong and T.~Yang.
\newblock On the volume conjecture for hyperbolic {D}ehn-filled $3 $-manifolds
  along the figure-eight knot.
\newblock {\em arXiv preprint arXiv:2003.10053}, 2020.

\bibitem[Yok96]{yok}
Y.~Yokota.
\newblock Topological invariants of graphs in 3-space.
\newblock {\em Topology}, 35(1):77--87, 1996.

\end{thebibliography}

%\address
\end{document}